\documentclass[twoside,leqno,10pt, A4]{amsart}
\usepackage{amsfonts}
\usepackage{amsmath}
\usepackage{amscd}
\usepackage{amssymb}
\usepackage{amsthm}
\usepackage{amsrefs}
\usepackage{latexsym}
\usepackage{mathrsfs}
\usepackage{bbm}
\usepackage{enumerate}
\usepackage{graphicx}

\usepackage{color}
\begin{document}

\newtheorem{theorem}[subsection]{Theorem}
\newtheorem{proposition}[subsection]{Proposition}
\newtheorem{lemma}[subsection]{Lemma}
\newtheorem{corollary}[subsection]{Corollary}
\newtheorem{conjecture}[subsection]{Conjecture}
\newtheorem{prop}[subsection]{Proposition}
\numberwithin{equation}{section}
\newcommand{\mr}{\ensuremath{\mathbb R}}
\newcommand{\mc}{\ensuremath{\mathbb C}}
\newcommand{\dif}{\mathrm{d}}
\newcommand{\intz}{\mathbb{Z}}
\newcommand{\ratq}{\mathbb{Q}}
\newcommand{\natn}{\mathbb{N}}
\newcommand{\comc}{\mathbb{C}}
\newcommand{\rear}{\mathbb{R}}
\newcommand{\prip}{\mathbb{P}}
\newcommand{\uph}{\mathbb{H}}
\newcommand{\fief}{\mathbb{F}}
\newcommand{\majorarc}{\mathfrak{M}}
\newcommand{\minorarc}{\mathfrak{m}}
\newcommand{\sings}{\mathfrak{S}}
\newcommand{\fA}{\ensuremath{\mathfrak A}}
\newcommand{\mn}{\ensuremath{\mathbb N}}
\newcommand{\mq}{\ensuremath{\mathbb Q}}
\newcommand{\half}{\tfrac{1}{2}}
\newcommand{\f}{f\times \chi}
\newcommand{\summ}{\mathop{{\sum}^{\star}}}
\newcommand{\chiq}{\chi \bmod q}
\newcommand{\chidb}{\chi \bmod db}
\newcommand{\chid}{\chi \bmod d}
\newcommand{\sym}{\text{sym}^2}
\newcommand{\hhalf}{\tfrac{1}{2}}
\newcommand{\sumstar}{\sideset{}{^*}\sum}
\newcommand{\sumprime}{\sideset{}{'}\sum}
\newcommand{\sumprimeprime}{\sideset{}{''}\sum}
\newcommand{\shortmod}{\ensuremath{\negthickspace \negthickspace \negthickspace \pmod}}
\newcommand{\V}{V\left(\frac{nm}{q^2}\right)}
\newcommand{\sumi}{\mathop{{\sum}^{\dagger}}}
\newcommand{\mz}{\ensuremath{\mathbb Z}}
\newcommand{\leg}[2]{\left(\frac{#1}{#2}\right)}
\newcommand{\muK}{\mu_{\omega}}
\newcommand{\sumn}{\sumstar_{(c,1+i)=1}  w\left( \frac {N(c)}X \right)}
\newcommand{\LL}{\mathcal{L}}

\makeatletter
\def\widebreve{\mathpalette\wide@breve}
\def\wide@breve#1#2{\sbox\z@{$#1#2$}%
     \mathop{\vbox{\m@th\ialign{##\crcr
\kern0.08em\brevefill#1{0.8\wd\z@}\crcr\noalign{\nointerlineskip}%
                    $\hss#1#2\hss$\crcr}}}\limits}
\def\brevefill#1#2{$\m@th\sbox\tw@{$#1($}%
  \hss\resizebox{#2}{\wd\tw@}{\rotatebox[origin=c]{90}{\upshape(}}\hss$}
\makeatletter

\title[Lower order terms of the
one level density of a family of quadratic Hecke $L$-functions]{Lower order terms of the
one level density of a family of quadratic Hecke $L$-functions}

\date{\today}
\author{Peng Gao and Liangyi Zhao}

\begin{abstract}
In this paper, we study lower order terms of the $1$-level density of low-lying zeros of quadratic Hecke L-functions in the Gaussian field. Assuming the Generalized Riemann Hypothesis, our result is valid for even test functions whose Fourier transforms are supported in $(-2, 2)$.  Moreover, we apply the ratios conjecture of $L$-functions to derive these lower order terms as well.  Up to the first lower order term, we show that our results is consistent with each other, when the Fourier transforms of the test functions are supported in $(-2, 2)$.
\end{abstract}

\maketitle



\section{Introduction}

 The work of H. L. Montgomery on the pair-correlation of zeros of $\zeta(s)$ in \cite{HM1} revealed, for the first time, the ties between zeros of $L$-functions and eigenvalues of random matrices. In more recent years, there has been a growing interest in the study on low-lying zeros of $L$-functions due to important roles they play in problems such as determining the rank of the Mordell-Weil groups of elliptic curves and the size of class numbers of imaginary quadratic
number fields. The relation between these low-lying zeros and the random matrices is predicted by the density conjecture of N. Katz and P. Sarnak \cites{KS1, K&S}, which asserts that the distribution of zeros near the central point of a family of $L$-functions is the same as that of the eigenvalues near $1$ of a corresponding classical compact group. \newline

   A rich literature exists on the density conjecture for various families of $L$-functions including the
Dirichlet \cites{Gao, O&S, Ru, Miller1, HuRu}, Hecke \cite{FI, M&P}, automorphic \cites{G&Zhao2, ILS, DuMi2, HuMi, RiRo, Royer}, elliptic curve \cites{SBLZ1, HB1, Brumer, SJM, Young}, Dedekind \cite{SST, Yang}, Artin \cite{ChoKim} and symmetric power $L$-functions \cite{Gu2, DM}. Among these various families, the investigation of the families of quadratic Dirichlet $L$-functions has a relative long history. It was first carried by  A. E. \"{O}zluk and C. Snyder in \cite{O&S} on the $1$-level density of low-lying zeros of the family, under the assumption of the Generalized Riemann Hypothesis (GRH). Further work in this direction can be found in  \cites{Gao, Ru, Miller1}. \newline

   The density conjecture predicts the main term behavior of the $n$-level density of low-lying zeros of families of $L$-functions for all $n$.  One can actually do more on the number theory side by computing the lower order terms of these $n$-level densities. These lower order terms serve to provide better understandings on the $n$-level densities.  Examples of such computations can be found in \cite{Young20015, RR, Miller2009}. \newline

   For the family of quadratic Dirichlet $L$-functions, the lower order terms of $1$-level density was first analyzed by S. J. Miller in \cite{Miller1}  for test functions whose Fourier transforms are supported in $(-1, 1)$. On the other hand, we note that the above-mentioned result of \"{O}zluk and Snyder \cite{O&S} on the $1$-level density is valid on GRH as long as the Fourier transforms of test functions are supported in $(-2, 2)$. Thus, one expects that the computation of the corresponding lower order terms for all such functions should be possible. This was indeed achieved by a recent result of D. Fiorilli, J. Parks and A. S\"{o}dergren in \cite{FPS} on GRH. \newline

  In \cite{G&Zhao4}, we studied the $1$-level density of low-lying zeros of quadratic Hecke $L$-functions in the Gaussian field. Assuming GRH, we showed that  our result confirms the density conjecture when the Fourier transforms of test functions are supported in $(-2, 2)$, a result analogous to that of the family of quadratic Dirichlet $L$-functions. In view of this, it is natural to ask whether one can compute the lower order terms as well, as in \cite{FPS}. We also point out here that in \cite{Waxman}, E. Waxman computed low order terms of the $1$-level density for a symplectic family of $L$-functions attached to Hecke characters of infinite order in the Gaussian field.  It is our goal in this paper to continue our work in this direction.  \newline

We shall write $K=\mq(i)$ for the Gaussian field and $\mathcal{O}_K=\mz[i]$ for the ring of integers in $K$. We write $N(n)$ for the norm of an element $n \in \mathcal{O}_K$ and we reserve the symbol $\chi_n$ for the quadratic Hecke character $\leg {n}{\cdot}$ defined in Section \ref{NTbackground}. We denote $\zeta_K(s)$ for the Dedekind zeta function of $K$. We assume GRH throughout this paper and we are concerned with the following family of $L$-functions:
$$ \mathcal F = \big\{ L(s,\chi_{i(1+i)^5c}) : c \hspace{0.1in} \text{square-free}, (c, 1+i)=1 \big\}.$$
  Let $L(s, \chi)$ be one of the $L$-functions in $\mathcal F$ and we write $\chi$ here in brief for the corresponding Hecke character. We denote the non-trivial zeroes of $L(s, \chi)$  by $1/2+i \gamma_{\chi, j}$ so that $\gamma_{\chi, j} \in \mr$ under GRH.  We order them as
\begin{equation*}
    \ldots \leq
    \gamma_{\chi, -2} \leq
    \gamma_{\chi, -1} < 0 \leq \gamma_{\chi, 1} \leq  \gamma_{\chi, 2} \leq
   \ldots.
\end{equation*}
   Let $X$ a large real number and we set $\mathcal{L}=\log X$ throughout the paper and we normalize the zeros by defining
\begin{align*}
    \tilde{\gamma}_{\chi, j}= \frac{\gamma_{\chi, j}}{2 \pi} \mathcal{L}.
\end{align*}

   Fix two even Schwartz class functions $\phi, w$ such that $w$ is non-zero and non-negative. We shall regard $\phi$ as a test function and $w$ as a weight function.  We define the $1$-level density for the single $L$-function $L(s, \chi)$ with respect to $\phi$ by the sum
\begin{equation*}
S(\chi, \phi)=\sum_{j} \phi(\tilde{\gamma}_{\chi, j}).
\end{equation*}
   The $1$-level density of the family $\mathcal F $ with respect to $w$ is then defined as the weighted sum
\begin{align}
\label{D}
  D(\phi;w, X) =\frac 1{W(X)}\sumstar_{c}  w\left( \frac {N(c)}X \right)  S(\chi_{i(1+i)^5c}, \phi),
\end{align}
    where we use $\sumstar$ to denote a sum over square-free elements in $\mathcal{O}_K$ throughout the paper and $W(X)$ here is the total weight given by
\begin{align*}
  W(X)=\sumstar_{c}  w\left( \frac {N(c)}X \right).
\end{align*}

  Our first result in this paper is an asymptotic expansion of $D(\phi;w, X)$ in descending powers of $\log X$, given in the following
\begin{theorem}
\label{quadraticmainthm}
Suppose that GRH holds for the family of $L$-functions in $\mathcal F $ as well as for $\zeta_K(s)$. Let $\phi(x)$ be an even Schwartz function whose
Fourier transform $\hat{\phi}(u)$ has compact support in $(-2,2)$ and let $w$ be an even non-zero and non-negative Schwartz function. Let $D(\phi;w, X)$ be defined as in \eqref{D}. Then we have for any integer $M \geq 1$,
\begin{align}
\label{Dexpansion}
D(\phi;w, X) =  \widehat \phi(0)-\frac{1}2\int\limits_{-1}^{1} \widehat \phi(u)\, \dif u +\sum_{m=1}^M \frac {R_{w,m}(\phi)}{\mathcal{L}^m} +O\left( \frac 1{\mathcal{L}^{M+1}}\right),
\end{align}
 where the coefficients $R_{w,m}(\phi)$ are linear functionals in $\phi$ that can be given explicitly in terms of $w$ and the derivatives of $\widehat\phi$ at the points $0$ and $1$ (see \eqref{Rwm}).
\end{theorem}

   We note that Theorem \ref{quadraticmainthm} gives a refinement of \cite[Theorem 1.1]{G&Zhao4}, which can be regarded as computing only the main term of the expansion for $\phi$ given in \eqref{Dexpansion}. Our result is similar to \cite[Theorem 1.1]{FPS}, and we follow closely many of the steps in \cite{FPS} in the proof of Theorem~\ref{quadraticmainthm}.  Additionally, our proof of Theorem ~\ref{quadraticmainthm}  proceeds along the same lines as that of \cite[Theorem 1.1]{G&Zhao4} with extra efforts to keep track of all the lower order terms. \newline

 When deriving these lower order terms, a powerful tool to deploy is the $L$-functions ratios conjecture of J. B. Conrey, D. W. Farmer and M. R. Zirnbauer in \cite[Section 5]{CFZ}. This approach was applied by J. B. Conrey and N. C. Snaith in \cite{CS} to study the $1$-level density function for zeros of quadratic Dirichlet $L$-functions. The general $n$-level density of the same family was carried out by A. M. Mason and N. C. Snaith in \cite{MS1}, and further enabled them to show in \cite{MS} that the result agrees with the density conjecture when the Fourier transforms of test functions are supported in $(-2, 2)$. \newline

  It is then highly desirable and interesting to compare the expressions for the $n$-level density functions conditional on the ratios conjecture to those obtained without the conjecture. For the family of quadratic Dirichlet $L$-functions, a result of S. J. Miller \cite{Miller1} matches the lower order terms of the $1$-level density function obtained with or without assuming the ratios conjecture, when the Fourier transforms of test functions are supported in $(-1, 1)$. When the support is enlarged to $(-2,2)$, D. Fiorilli, J. Parks and A. S\"{o}dergren obtained the lower order terms of the $1$-level density function in \cite{FPS, FPS1} by applying either the ratios conjecture or otherwise. Their work assumes GRH and the results obtained are further shown to match up to the first lower order term in \cite{FPS1}. \newline

Motivated by the above works, our next objective in the paper is to evaluate $D(\phi;w, X)$ using the ratios conjecture.  We shall formulate the appropriate version of the ratios conjecture concerning our family $\mathcal F$ in Conjecture  \ref{ratiosconjecture}, and use it to prove in Section \ref{The ratios conjecture's prediction} the following asymptotic expression of $D(\phi;w, X)$.
\begin{theorem}
\label{oneleveldensityresult}
 Assume the truth of GRH for the family of $L$-functions in $\mathcal F $ as well as for $\zeta_K(s)$ and Conjecture \ref{ratiosconjecture}. Let $w(t)$ be an even,
non-zero and non-negative Schwartz function and $\phi(x)$ an even Schwartz function whose Fourier transform $\widehat{\phi}(u)$ has
compact support, then for any $\varepsilon > 0$,
\begin{align}
\label{Drationconj}
\begin{split}
D(\phi;w, X)=\frac{1}{W(X)} & \sumn \frac{1}{2\pi}
\int\limits_{\mathbb R} \bigg(2 \frac{\zeta'_K(1+2it)}{\zeta_K(1+2it)} + 2A_{\alpha}(it,it)+\log \left(\frac {32N(c)}{\pi^2}\right) +\frac{\Gamma'}{\Gamma}\left(\frac 12-it\right)\\
&+ \frac{\Gamma'}{\Gamma}\left(\frac 12+it  \right)-\frac {8}{\pi}  X_{c}\left(\frac{1}{2}+it\right)\zeta_K(1-2it)A(-it,it) \bigg) \,
\phi\left(\frac{t\LL}{2\pi}\right) \, \dif t+ O_{\varepsilon}\big(X^{-1/2+\varepsilon}\big),
\end{split}
\end{align}
where the functions $X_c$, $A$ and $A_{\alpha}$ are given in \eqref{xe}, \eqref{defnofae} and Lemma~\ref{ratiostheorem}, respectively.
\end{theorem}

 Our next goal is then to compare the expression given for $D(\phi;w, X)$ in Theorem \ref{oneleveldensityresult} with the one obtained in Theorem~\ref{quadraticmainthm}.  To this end, we will prove (see Lemma~\ref{lemma explicit formula all d}), the following expression for $D(\phi;w, X)$ when $\phi$ is an even Schwartz test function with compactly supported Fourier transform.
\begin{align}
\label{Sstar}
\begin{split}
D(\phi;w, X) =  &\frac {\widehat \phi(0)}{ \LL W(X)}\underset{(c, 1+i) =1 }{\sum \nolimits^{*}} w\left( \frac {N(c)}X \right)  \log N(c) +\frac {\widehat\phi(0)}{\LL} \left( \log\frac {32}{\pi^2}+ 2\frac{\Gamma'}{\Gamma} \left( \frac 12 \right) \right ) \\
&\hspace*{1cm}  -\frac 2{\LL W(X)} \underset{(c, 1+i) =1 }{\sum \nolimits^{*}}  w\left( \frac {N(c)}X \right) \sum_{j \geq 1} S_j(\chi_{i(1+i)^5c}, \LL ; \hat{\phi}) +\frac{2}{\LL}\int\limits_0^\infty\frac{e^{-x/2}}{1-e^{-x}}\left(\widehat{\phi}(0)-\widehat{\phi}\left(\frac{x}{\LL}\right)\right) \dif x,
\end{split}
\end{align}
 where
\begin{equation*}
   S_j(\chi_{i(1+i)^5c}, \LL ;\hat{\phi})=\sum_{\varpi \equiv 1 \bmod {(1+i)^3}}\frac { \log
   N(\varpi)}{\sqrt{N(\varpi^j)}}\chi_{i(1+i)^5c} \left( \varpi^j \right)  \hat{\phi}\left( \frac {\log N(\varpi^j)}{\LL} \right)
\end{equation*}
with the sum over $\varpi$ running over primes in $\mathcal{O}_K$. Here we note that in $\mathcal{O}_K$, every ideal co-prime to $(1+i)$ has a unique generator congruent to $1$ modulo $(1+i)^3$ and such a generator is called primary. We shall use $\sum_{\varpi \equiv 1 \bmod {(1+i)^3}}$ (or sum over other variables) to indicate a sum over primary elements in $\mathcal{O}_K$. \newline

   For any function $W$, we denote its Mellin transform by $ \mathcal M W$, so that
\begin{align}
\label{MW}
      \mathcal M W(s) =\int\limits^{\infty}_0W(t)t^s\frac {\dif t}{t}.
\end{align}
   We further note that around $s=1$,
\begin{align}
\label{zetaKexpan}
  \zeta_K(s)=\frac {\pi}{4}\cdot \frac 1{s-1}+\gamma_K+O(|s-1|),
\end{align}
  where $\gamma_K$ is a constant. We write $\gamma=0.57\cdots$ for the Euler constant. \newline

  Our following result shows the agreement of the two expressions for $D(\phi;w, X)$ given in \eqref{Drationconj} and \eqref{Sstar} up to the first lower order term.
\begin{theorem}
\label{quadraticmainthmratio}
 Assume  the truth of GRH for the family of $L$-functions in $\mathcal F $ as well as for $\zeta_K(s)$ and Conjecture \ref{ratiosconjecture}.  Let $w(t)$ be an even,
non-zero and non-negative Schwartz function and $\phi(x)$ an even Schwartz function whose Fourier transform $\widehat{\phi}(u)$ has
compact support. Then the expression \eqref{Drationconj} gives that
\begin{align}
\label{Dexpansionratio}
\begin{split}
D(\phi;w, X) =\widehat \phi(0) &+ \int\limits^{\infty}_1 \widehat \phi(\tau) \ \dif \tau+ \frac{\widehat \phi(0)}{\LL}\left( \log \frac{32}{\pi^2} +2\frac{\Gamma'}{\Gamma}\left(\frac 12\right)+ \frac 2{\widehat w(0)}\int\limits_0^{\infty} w(x) \log x \ \dif x   \right) \\
& +\frac{2}{\LL} \int\limits_0^\infty\frac{e^{-t/2}}{1-e^{-t}}\left(\widehat{\phi}(0)-\widehat{\phi}\left(\frac{t}{\LL}\right)\right) \dif t \\
& - \frac{2}{\LL} \sum_{\substack{\varpi \equiv 1 \bmod {(1+i)^3} \\ j\geq 1}} \frac{\log N(\varpi)}{N(\varpi)^{j}} \left(  1+\frac 1{N(\varpi)} \right)^{-1} \widehat \phi\left( \frac{2j \log N(\varpi)}{\LL} \right) \\
& +\frac {\widehat \phi(1)}{\LL} \left( 2\gamma+ \log \left( \frac {\pi^2}{2^{7/3}} \right)+2\frac{\zeta'_K(2)}{\zeta_K(2)} -\frac {8}{\pi} \gamma_K -\frac{\mathcal M w'(1)}{\mathcal Mw(1)} \right)  +O\left( \LL^{-2} \right).
\end{split}
\end{align}
 Also, when $\text{sup}(\text{supp}(\hat{\phi}(u)))<2$, the above expression agrees with that given in \eqref{Sstar}.
\end{theorem}

   We give the proof of Theorem \ref{quadraticmainthmratio} in Section \ref{Compartionofresults}. Our approach is inspired by the proof of  \cite[Theorems 1.1 and 1.4]{FPS1}, although the computation in our situation is more involved.

\section{Preliminaries}
\label{sec 2}

\subsection{Number fields background}
\label{NTbackground}

   Recall that in this paper, $K=\mq(i)$ is the Gaussian field. We denote $U_K$ for the group of units in ${\mathcal O}_K$, so that $U_K=\{ \pm 1, \pm i\}$. As it is well-known that $K$ has class number one, we shall not distinguish $n$ and $(n)$ when this causes no confusion from the context.  We therefore write $\mu_{[i]}(n)$ to mean the M\"obius function $\mu_{[i]}((n))$.  We shall use $\varpi$ to denote a prime (or prime ideal) in $K$ and $\Lambda(n)$ is the von Mangoldt function on $\mathcal{O}_K$ so that
\begin{align*}
    \Lambda(n)=\begin{cases}
   \log N(\varpi) \qquad & n=\varpi^k, \; \text{$\varpi$ prime}, \; k \geq 1 \\ \\
     0 \qquad & \text{otherwise}.
    \end{cases}
\end{align*}

Let $\leg {\cdot}{n}_4$ stand for the quartic
residue symbol on $\mathcal{O}_K$.  For a prime $\varpi \in \mathcal{O}_K$
with $N(\varpi) \neq 2$, the quartic symbol is defined for $a \in
\mathcal{O}_K$, $(a, \varpi)=1$ by $\leg{a}{\varpi}_4 \equiv
a^{(N(\varpi)-1)/4} \pmod{\varpi}$, with $\leg{a}{\varpi}_4 \in \{
\pm 1, \pm i \}$. When $\varpi | a$, we define
$\leg{a}{\varpi}_4 =0$.  Then the quartic symbol can be extended
to any composite $n$ with $(N(n), 2)=1$ multiplicatively. We further define $(\frac{\cdot}{n})=\leg {\cdot}{n}^2_4$  to be the quadratic
residue symbol for these $n$. \newline

    We say an element $c \in \mathcal{O}_K$ (or the ideal $(c)$) is odd if $(c, 1+i)=1$.  Note that in $\mathcal{O}_K$, every odd ideal has a unique generator congruent to $1$ modulo $(1+i)^3$.  Such a generator is called primary.  For two primary integers $m, n \in \mathcal{O}_K$, we note that the quadratic reciprocity law (see \cite[(2.1)]{G&Zhao4}) gives
\begin{align}
\label{quadrecip}
 \leg{m}{n} = \leg{n}{m}.
\end{align}

Let $\chi$ denote a Hecke character of $K$ and we say that $\chi$ is of trivial infinite type if its component at the infinite place of $K$ is trivial. In particular, $\chi_c$ defined earlier is a Hecke character of trivial infinite type. It is further shown in \cite[Section 2.1]{G&Zhao4} that when $c$ is square-free and co-prime to $1+i$, $\chi_{i(1+i)^5c}$ defines a primitive Hecke character $\text{mod} ((1+i)^5c)$ of trivial infinite type. \newline

    For any primitive Hecke character $\chi \pmod {m}$ of trivial infinite type,  let
    \[ \Lambda(s, \chi) = (|D_K|N(m))^{s/2}(2\pi)^{-s}\Gamma(s)L(s, \chi), \]
    where $L(s, \chi)$ is the the $L$-functions attached to $\chi$ and $D_K=-4$ is the discriminant of $K$. In particular, we use $\zeta_K(s)$ to denote the Dedekind zeta function of $K$. \newline

    It was shown by E. Hecke that $\Lambda(s, \chi)$ is an entire function and satisfies the following functional equation (\cite[Theorem 3.8]{iwakow})
\begin{align*}
  \Lambda(s, \chi) = W(\chi)(N(m))^{-1/2}\Lambda(1-s, \overline{\chi}),
\end{align*}
   where $|W(\chi)|=(N(m))^{1/2}$.

\subsection{Poisson Summation}

    For any $r,n \in \mathcal{O}_K$ with $n$ odd, we define the Gauss sum $g(r,n)$ as
\begin{align*}
 g(r,n) = \sum_{x \bmod{n}} \leg{x}{n} \widetilde{e}\leg{rx}{n}
\end{align*}
   where $ \widetilde{e}(z) =\exp \left( 2\pi i  (\frac {z}{2i} - \frac {\overline{z}}{2i}) \right)$.  It is shown in \cite[Lemma 2.2]{G&Zhao4} that, for a primary prime $\varpi$,
\[  g(r, \varpi)=\leg {ir}{\varpi}N(\varpi)^{1/2}. \]

    We quote the following Poisson summation formula from \cite[Lemma 2.7]{G&Zhao4}.
\begin{lemma}
\label{Poissonsum} Let $n \in \mathcal{O}_K$ be primary and $\chi$ a quadratic character $\pmod {n}$ of trivial infinite type. For any Schwartz class function $W$,  we have
\begin{equation} \label{quadpoi}
   \sum_{m \in \mathcal{O}_K}\chi(m)W\left(\frac {N(m)}{X}\right)=\frac {X}{N(n)}\sum_{k \in
   \mathcal{O}_K}g(k,n)\widetilde{W}\left(\sqrt{\frac {N(k)X}{N(n)}}\right)
\end{equation}
and
\begin{equation} \label{tripoi}
   \sum_{m \in \mathcal{O}_K} W\left(\frac {N(m)}{X}\right)=X \sum_{k \in
   \mathcal{O}_K}\widetilde{W}\left(\sqrt{N(k)X}\right),
\end{equation}
where
\begin{align}
\label{wtilde}
   \widetilde{W}(t) &=\int\limits^{\infty}_{-\infty}\int\limits^{\infty}_{-\infty}W(N(x+yi))\widetilde{e}\left(- t(x+yi)\right)\dif x \dif y, \quad t \geq 0.
\end{align}
\end{lemma}

   We include here our conventions for various transforms used in this paper. For any function $W$, we write $\widehat{W}$ for the Fourier transform of $W$ and recall that its Mellin transform $ \mathcal M w$ is defined in \eqref{MW}. We note that, for any Schwartz class function $W$ and any integer $E \geq 0$, integration by parts $E+1$ times yields that for $\Re(s)>0$,
\begin{equation}
\label{eq:h}
 \mathcal M  {W}(s) \ll \frac{1}{|s| (1+|s|)^{E}}.
\end{equation}

    Furthermore, we set
\begin{align*}
   \widebreve{W}(t) =\int\limits^{\infty}_{-\infty}\int\limits^{\infty}_{-\infty}\widetilde{W}(N(u+vi))\widetilde{e}\left(- t(u+vi)\right)\dif u \dif v, \quad t \geq 0.
\end{align*}

   When $W(t)$ is a real smooth function, one follows the arguments that lead to the bounds given in \cite[(2.12)]{G&Zhao4} to get that both $\widetilde{W}$
   and $\widebreve{W}$ are real and
\begin{align}
\label{bounds}
     \widetilde{W}^{(\mu)}(t), \widebreve{W}^{(\mu)}(t)  \ll_{j} \min \{ 1, |t|^{-j} \}
\end{align}
    for all integers $\mu \geq 0$, $j \geq 1$ and all real $t$.

\subsection{Some consequences of GRH}
   In this section we state a few results that are derived by using GRH. The first one is for sums over primes.
\begin{lemma}
\label{lem2.7}
Suppose that GRH is true. For any Hecke character $\chi \pmod {m}$ of trivial infinite type, we have for $X \geq 1$,
\begin{align} \label{lem2.7eq}
\sum_{\substack {N(\varpi) \leq X \\ \varpi \equiv 1 \bmod {(1+i)^3}}} \chi (\varpi) \log N(\varpi)
=\delta_{\chi} X+ O\big( X^{1/2} \log^{2} (2X) \log N(m) \big ),
\end{align}
    where $\delta_{\chi}=1$ if $\chi$ is principal and $\delta_{\chi}=0$ otherwise.  Moreover, we have
\begin{align}
\label{mer}
  \sum_{\substack{ N(\varpi) \leq X \\ \varpi \equiv 1 \bmod {(1+i)^3}}} \frac {\log N(\varpi)}{N(\varpi) }= \log X+O(1).
\end{align}
\end{lemma}
\begin{proof}
   The formula in \eqref{lem2.7eq} follows directly from \cite[Theorem 5.15]{iwakow} and \eqref{mer} is derived from \eqref{lem2.7eq} by taking $\chi$ to be the principal character modulo $1$ and using partial summation.
\end{proof}

    Our next two lemmas provide estimations on certain weighted quadratic character sums. The following one is a generalization of \cite[Lemma 2.10]{FPS16}.
\begin{lemma}
\label{lemma count of squarefree}
  Suppose that GRH is true.  For any even, non-zero and non-negative Schwartz function $w$, we have for any primary  $n \in \mathcal{O}_K$ and $\epsilon>0$,
\begin{equation*}
\underset{(c, 1+i) =1 }{\sum \nolimits^{*}}  w\left( \frac {N(c)}X \right) \left(\frac {i(1+i)^5c}{n} \right)= \delta_{\chi_n }\frac{\pi X}{3\zeta_K(2)} \widehat w(0)   \prod_{\varpi\mid n}  \left(1+\frac 1{N(\varpi)} \right)^{-1}+O\big(N(n)^{3(1-\delta_{\chi_n })/8 +\varepsilon}X^{1/4+\varepsilon}\big).
\end{equation*}
  Here we recall that $\sum^*$ denotes the sum over square-free elements in $\mathcal{O}_K$.
\end{lemma}

\begin{proof}
    Since each $c$ co-prime to $1+i$ can be uniquely written as the product of a unit and a primary element, it follows that
\begin{align*}
\underset{(c, 1+i) =1 }{\sum \nolimits^{*}}  w\left( \frac {N(c)}X \right) \left(\frac {(1+i)^5c}{n} \right)
=& \left(1+ \leg {-1}{n}+\leg {i}{n}+\leg {i^3}{n} \right) \sumstar_{c \equiv 1 \mod (1+i)^3 } w\left( \frac {N(c)}X \right) \left(\frac {(1+i)c}{n}\right) \\
=&  2\left(1+ \leg {i}{n} \right)\left(\frac {1+i}{n} \right) \sumstar_{c \equiv 1 \mod (1+i)^3 } w\left( \frac {N(c)}X \right) \left(\frac {c}{n} \right).
\end{align*}

   Note further that the quadratic reciprocity law \eqref{quadrecip} allows us to write $\chi_{n}(c)$ for $\leg {c}{n}$ and $\chi_n$ is a Hecke character.  We then apply the Mellin inversion formula and get
\begin{align*}
  & \sumstar_{c \equiv 1 \mod (1+i)^3 }  w\left( \frac {N(c)}X \right)\chi_{n}(c)  = \frac 1{2\pi i } \int_{(2)} \sum_{\substack{c \equiv 1 \mod (1+i)^3}} \frac{\mu_{[i]}^2(c) \chi_{n}(c) }{N(c)^s} X^s  \mathcal M w(s) \dif s \\
 &= \frac 1{2\pi i } \int\limits_{(2)} \prod_{\substack{\varpi \equiv 1 \mod (1+i)^3 \\ (\varpi, 1+i)=1}} \bigg(1+ \frac{\chi_{n}(\varpi)}{N(\varpi)^s} \bigg)  \mathcal M w (s)X^s \dif s = \frac 1{2\pi i } \int\limits_{(2)}  \frac{L(s, \chi_n)}{L(2s, \chi^2_n)}\bigg(1+ \frac{\chi_{n}(1+i)}{N(1+i)^s} \bigg)^{-1} \mathcal M w (s)X^s \dif s.
\end{align*}
Here and after, we write $\int_{(c)}$ for the integral over the vertical line with $\Re(s) = c$. \newline

  We shift the line of integration to $\Re(s)=1/4+\varepsilon$ and we encounter a pole at $s=1$ only when $\chi_n$ is a principal character. In that case, the residue is easily seen (recall that the residue of $\zeta_K(s)$ at $s = 1$ is $\pi/4$) to be
\begin{align*}
   \frac {\pi}{6}\zeta^{-1}_K(2) \prod_{\varpi | n} \left( 1+\frac 1{N(\varpi)} \right)^{-1}  \mathcal M w  (1) X= \frac {\pi}{12}\zeta^{-1}_K(2) \prod_{\varpi | n} \left( 1+\frac 1{N(\varpi)} \right)^{-1}  \widehat w  (0) X.
\end{align*}
   by noting that $\mathcal M w  (1)=\widehat w  (0)/2 $ when $w$ is even. The remaining integral over the line $\Re(s)=1/4+\varepsilon$  can be estimated by using \eqref{eq:h} for a suitable $E$ and the bound that assuming GRH,
\begin{align}
\label{zetainversebound}
   L^{-1}(2s, \chi_n) \ll (N(n)(1+\Im(s))^{\varepsilon},
\end{align}
  which follows from  \cite[Theorem 5.19]{iwakow}. This gives the result when $\chi_n$ is a principal character. \newline

  When $\chi_n$ is not principal, we apply the convexity bound \cite[(5.20)]{iwakow} for $L$-functions attached to non-principal characters, such that
$$ L(s,\chi_n) \ll_{\varepsilon} (N(n)(|s|+1)^2)^{(1-\Re(s))/2+\varepsilon/2} \hspace{1cm} (0\leq \Re(s) \leq 1),  $$
   Combining this with \eqref{zetainversebound} allows us to readily deduce the assertion of the lemma for $\chi_n$ being non-principal. This completes the proof.
\end{proof}

   By taking $n=1$ in Lemma \ref{lemma count of squarefree}, we immediately obtain that
\begin{equation}
\label{W}
W(X)= \frac{\pi X}{3\zeta_K(2)} \widehat w(0)   +O \left( X^{1/4+\varepsilon} \right).
\end{equation}

\begin{lemma}
\label{lemma logd}
Suppose that GRH is true. For any even, non-zero and non-negative Schwartz function $w$, we have
\begin{align*}
\frac {1}{ W(X)}\underset{(c, 1+i) =1 }{\sum \nolimits^{*}}  w\left( \frac {N(c)}X \right)  \log N(c)= \log X + \frac 2{\widehat w(0)}\int\limits_0^{\infty} w(x) \log x \ \dif x  +O \left( X^{-1/2 +\varepsilon} \right).
\end{align*}
\end{lemma}
\begin{proof}
  We have
\begin{align*}
\underset{(c, 1+i) =1 }{\sum \nolimits^{*}}  w\left( \frac {N(c)}X  \right)  \log N(c)
=& -\frac 4{2\pi i}\int\limits_{(2)} \frac{\dif}{\dif s} \left( \sum_{ \substack{c \equiv 1 \mod (1+i)^3}} \frac{\mu^2_{[i]}(c)}{N(c)^s}  \right) \mathcal M w(s)X^s \ \dif s \\
=& -\frac 4{2\pi i}\int\limits_{(2)} \frac{\dif}{\dif s} \left( \frac{\zeta_K(s)}{\zeta_K(2s) ( 1+2^{-s} )} \right) \mathcal M w(s)X^s \ \dif s.
\end{align*}
We shall shift the contour of integration to the line $\Re(s)=1/4+\varepsilon$. Note that under GRH, the only poles of the function
\begin{align*}
 \frac{\dif}{\dif s} \left( \frac{\zeta_K(s)}{\zeta_K(2s) ( 1+2^{-s} )} \right)  = \frac{\zeta'_K(s)}{\zeta_K(2s) ( 1+2^{-s} )}- \frac{2\zeta'_K(2s)\zeta_K(s)}{\zeta^{2}_K(2s) ( 1+2^{-s} )}+\frac{\log 2 \cdot \zeta_K(s)}{2^{s} \zeta_K(2s) ( 1+2^{-s} )^{2} }
\end{align*}
in the region $1/4+\varepsilon \leq \Re(s)\leq 2$ are at $s=1$ and $s=1/2$. Only the contribution of the residue at $s=1$ is $\gg X$. It is then easy to compute the contribution of the residues to be
\begin{align*}
   \frac{2\pi }{3\zeta_K(2)} \mathcal M w(1)X \log X+\frac{2\pi  }{3\zeta_K(2)}  (\mathcal M w)'(1)X+O(X^{1/2+\varepsilon}).
\end{align*}
   The assertion of the lemma follows from this and \eqref{W}, by noting that $\mathcal M w(1)=\widehat {w}(0)/2$.
\end{proof}

\subsection{The Explicit Formula}

Let $f$ be an even, positive Schwartz function whose Fourier transform $\hat{f}$ is a smooth function with compact support.
 Let $\chi$ be a primitive Hecke character $\chi \pmod {m}$ of trivial infinite type.
In this section we derive an explicit formula which allows us to convert the evaluation of $f$ at the non-trivial zeros of $L(s,\chi)$ for to a sum over powers of prime ideals. Note that the non-trivial zeros of $L(s,\chi)$ are precisely those of the corresponding $\Lambda(s, \chi)$.   For some $c >1$, consider the following integral
\begin{equation*}
     \frac 1{2 \pi i}\int\limits_{(c)}\frac {\Lambda'}{\Lambda}(s, \chi)f \left( \frac {s-1/2}{2 \pi i} \right) \dif s.
\end{equation*}
By moving the line of integration to $1-c$, we obtain
\begin{equation*}
     \frac 1{2 \pi i}\int\limits_{(c)}\frac {\Lambda'}{\Lambda}(s, \chi)f \left( \frac {s-1/2}{2 \pi i} \right) \dif s = \sum_{j}f \left( \frac {\gamma_{\chi,j}}{2 \pi} \right)+
     \frac 1{2 \pi i}\int\limits_{(1-c)}\frac {\Lambda'}{\Lambda}(s, \chi)f \left( \frac {s-1/2}{2 \pi i} \right) \dif s.
\end{equation*}
    Now the functional equation $\Lambda(s, \chi)=W(\chi)(N(m))^{-1/2}\Lambda(1-s, \overline{\chi})$ implies
\begin{align*}
   \frac {\Lambda'}{\Lambda}(s, \chi)=-\frac {\Lambda'}{\Lambda}(1-s, \overline{\chi}).
\end{align*}

    It follows that
\begin{align*}
   \sum_{j}f \left( \frac {\gamma_{\chi,j}}{2 \pi} \right)=\frac 1{2 \pi i}\int\limits_{(c)}\frac {\Lambda'}{\Lambda}(s, \chi)
   f \left( \frac {s-1/2}{2 \pi i} \right) \dif s+\frac 1{2 \pi i}\int\limits_{(c)}\frac {\Lambda'}{\Lambda}(s, \overline{\chi})
   f\left( \frac {1/2-s}{2 \pi i} \right) \dif s.
\end{align*}
    Using
\begin{equation*}
    \frac {\Lambda'}{\Lambda}(s, \chi)=\frac 1{2}\log \frac {|D_K|N(m)}{(2\pi)^2}+\frac
    {\Gamma'}{\Gamma}(s)+\frac {L'}{L}(s, \chi),
\end{equation*}
     we obtain that
\begin{align*}
    \frac 1{2 \pi i}\int_{(c)}\frac {\Lambda'}{\Lambda}(s, \chi)
   f \left( \frac {s-1/2}{2 \pi i} \right) \dif s+\frac 1{2 \pi i}\int_{(c)}\frac {\Lambda'}{\Lambda}(s, \overline{\chi})
   f\left( \frac {1/2-s}{2 \pi i} \right) \dif s =T_1+T_2,
\end{align*}
   where
\[   T_1 = \frac 1{2 \pi i}\int\limits_{(c)}\left( \frac 1{2}\log \frac {|D_K|N(m)}{(2\pi)^2}+\frac
    {\Gamma'}{\Gamma}(s)\right) f \left( \frac {s-1/2}{2 \pi i} \right) \dif s +\frac 1{2 \pi i}\int\limits_{(c)}\left( \frac 1{2}\log \frac {|D_K|N(m)}{(2\pi)^2}+\frac
    {\Gamma'}{\Gamma}(s)\right) f \left( \frac {1/2-s}{2 \pi i} \right) \dif s, \]
    and
    \[    T_2 =  \frac 1{2 \pi i}\int\limits_{(c)}\frac {L'}{L}(s, \chi)f \left( \frac {s-1/2}{2 \pi i} \right) \dif s+\frac 1{2 \pi i}\int\limits_{(c)}\frac {L'}{L}(s, \overline{\chi})f \left( \frac {1/2-s}{2 \pi i} \right) \dif s. \]
    For $T_1$, we move the line of integration to $1/2$ and apply the change of variables $s=1/2+2\pi i t$.  In so doing, we obtain
\begin{align}
\label{I1}
\begin{split}
    T_1 &=\int\limits^{\infty}_{-\infty}\left( \log \frac {|D_K|N(m)}{(2\pi)^2}+ \frac
    {\Gamma'}{\Gamma} \left( \frac{1}{2}+2\pi i t \right)+ \frac
    {\Gamma'}{\Gamma} \left( \frac{1}{2}-2\pi i t \right) \right) f(t) \dif t  \\
    &=\int\limits^{\infty}_{-\infty}\left( \log\frac {|D_K|N(m)}{(2\pi)^2}+ 2\frac
    {\Gamma'}{\Gamma} \left( \frac 12 \right) \right )f(t) \dif t+\int\limits^{\infty}_{0}\frac {e^{-t/2}}{1-e^{-t}} \left( 2\hat{f}(0)-\hat{f}(t)-\hat{f}(-t) \right)  \dif t,
\end{split}
\end{align}
where the second equality above follows from \cite[Lemma 12.14]{MVa1}. \newline

We express $T_2$ as
\begin{align*}
    T_2=-\sum_{\substack{(n) \\ n \in \mathcal{O}_K}} \chi(n)\Lambda(n)\frac 1{2 \pi i}\int\limits_{(c)}\frac 1{N(n)^s}f \left( \frac {s-1/2}{2 \pi i} \right) \dif s-\sum_{\substack{(n) \\ n \in \mathcal{O}_K}} \overline{\chi}(n) \Lambda(n)\frac 1{2 \pi i}\int\limits_{(c)}\frac 1{N(n)^s}f \left( \frac {1/2-s}{2 \pi i} \right) \dif s.
\end{align*}

    Moving the lines of integration for $I_2$ to $\Re(s)=1/2$ and
    setting $s=1/2+2\pi i t$, the integrations become
\begin{equation*}
    \int\limits^{\infty}_{-\infty}\frac 1{N(n)^{1/2 \pm 2 \pi i t}}f(t)dt=\frac 1{\sqrt{N(n)}}\hat{f}(\mp \log N(n)).
\end{equation*}

Thus
\begin{align}
\label{I2}
    T_2=-\sum_{\substack{(n) \\ n \in \mathcal{O}_K}} \frac {\chi(n)\Lambda(n)}{\sqrt{N(n)}}\hat{f}(-\log N(n))-\sum_{\substack{(n) \\ n \in \mathcal{O}_K}} \frac {\overline{\chi}(n)\Lambda(n)}{\sqrt{N(n)}}\hat{f}( \log N(n)).
\end{align}

    We then derive from \eqref{I1} and \eqref{I2} that
\begin{align*}
    \sum_{j}f \left( \frac {\gamma_{\chi,j}}{2 \pi} \right) =& \int\limits^{\infty}_{-\infty}\left( \log\frac {|D_K|N(m)}{(2\pi)^2}+ 2\frac
    {\Gamma'}{\Gamma}\left( \frac 12 \right) \right )f(t) \dif t+\int\limits^{\infty}_{0}\frac {e^{-t/2}}{1-e^{-t}}(2\hat{f}(0)-\hat{f}(t)-\hat{f}(- t)) \dif t \\
    & \hspace*{1cm} -\sum_{\substack{(n) \\ n \in \mathcal{O}_K}} \frac {\chi(n)\Lambda(n)}{\sqrt{N(n)}}\hat{f}(-\log N(n))-\sum_{\substack{(n) \\ n \in \mathcal{O}_K}} \frac {\overline{\chi}(n)\Lambda(n)}{\sqrt{N(n)}}\hat{f}( \log N(n)).
\end{align*}

    Now an easy computation in Fourier transform gives
\begin{align*}
   \sum_{j}f \left( \frac {\log X \gamma_{\chi,j}}{2 \pi} \right) =& \frac 1{\log X} \int\limits^{\infty}_{-\infty}\left( \log\frac {|D_K|N(m)}{(2\pi)^2}+ 2\frac
    {\Gamma'}{\Gamma} \left( \frac 12 \right) \right )f(t) \dif t \\
    &\hspace*{1cm} +\frac 1{\log X}\int\limits^{\infty}_{0}\frac {e^{-t/2}}{1-e^{-t}}(2\hat{f}(0)-\hat{f} \left( \frac {t}{\log X} \right)-\hat{f}\left( -\frac {t}{\log X}) \right) \dif t
     \\
   & \hspace*{1cm}  -\frac 1{\log X}\sum_{\substack{(n) \\ n \in \mathcal{O}_K}} \frac {\chi(n)\Lambda(n)}{\sqrt{N(n)}}\hat{f} \left( -\frac {\log N(n)}{\log X} \right)-\frac 1{\log X}\sum_{\substack{(n) \\ n \in \mathcal{O}_K}} \frac {\overline{\chi}(n)\Lambda(n)}{\sqrt{N(n)}}\hat{f} \left( \frac {\log N(n)}{\log X} \right).
\end{align*}
Recall that $\mathcal{L}=\log X$ and note that as $f$ is even, so is $\hat{f}$. Thus, when $\chi$ is a quadratic Hecke character, we can simplify
   the above expression as
\begin{align}
\label{fef}
\begin{split}
   \sum_{j}f\left( \frac {\mathcal{L} \gamma_{\chi,j}}{2 \pi} \right) =& \frac 1{\mathcal{L}}\left( \log\frac {|D_K|N(m)}{(2\pi)^2}+ 2\frac
    {\Gamma'}{\Gamma}\left( \frac 12 \right) \right )\hat{f}(0)+\frac 2{\mathcal{L}}\int^{\infty}_{0}\frac {e^{-t/2}}{1-e^{-t}}\left( \hat{f}(0)-\hat{f}\left(\frac {t}{\mathcal{L}}\right) \right) \dif t
     \\
   & \hspace*{1cm}  -\frac 2{\mathcal{L}}\sum_{\substack{(n) \\ n \in \mathcal{O}_K}} \frac {\chi(n)\Lambda(n)}{\sqrt{N(n)}}\hat{f} \left( \frac {\log N(n)}{\mathcal{L}} \right).
\end{split}
\end{align}

   Recall that every odd prime $\varpi \in \mathcal{O}_K$ has a primary generator. In our paper, we work on explicitly with the Hecke characters $\chi_{i(1+i)^5c}$ for odd, square-free $c$. Our choice for such characters is inspired by the Dirichlet characters $\chi_{8d}$ for odd, square-free rational integers $d$ considered by K. Soundararajan \cite{sound1} in his work on non-vanishing of quadratic Dirichlet $L$-functions at the central value. The advantage of using the characters $\chi_{i(1+i)^5c}$ is that, besides their primitivity, the presence of the factors $i(1+i)^3$ makes the resulting expression much neater after applying the Poisson summation formula given in Lemma \ref{Poissonsum}. We now apply the formula given in \eqref{fef} to the special case $\chi_{i(1+i)^5c}$ for odd, square-free $c$ to arrive at the following
\begin{lemma}[Explicit Formula]
\label{lem2.4}
   Let $\phi(x)$ be an even Schwartz function whose Fourier transform
   $\hat{\phi}(u)$ has compact support. Let $c \in \mathcal{O}_K$ be square-free satisfying $(c, 1+i)=1$. We have
\begin{align*}
\begin{split}
S(\chi_{i(1+i)^5c}, \phi) =& \frac 1{\mathcal{L}}\left( \log\frac {32N(c)}{\pi^2}+ 2\frac
    {\Gamma'}{\Gamma} \left( \frac 12 \right) \right )\hat{\phi}(0)+\frac 2{\mathcal{\mathcal{L}}}\int\limits^{\infty}_{0}\frac {e^{-t/2}}{1-e^{-t}}\left( \hat{\phi}(0)-\hat{\phi}\left( \frac {t}{\mathcal{\mathcal{L}}} \right) \right) \dif t
      -\frac 2{\mathcal{L}}\sum_{j \geq 1} S_j(\chi_{i(1+i)^5c}, \mathcal{L} ;\hat{\phi}),
\end{split}
\end{align*}
   where
\begin{equation*}
   S_j(\chi_{i(1+i)^5c}, \mathcal{L};\hat{\phi})=\sum_{\varpi \equiv 1 \bmod {(1+i)^3}}\frac { \log
   N(\varpi)}{\sqrt{N(\varpi^j)}}\chi_{i(1+i)^5c} \left( \varpi^j \right)  \hat{\phi}\left( \frac {\log N(\varpi^j)}{\mathcal{L}} \right)
\end{equation*}
with the sum over $\varpi$ running over primes in $\mathcal{O}_K$.
\end{lemma}

Using Lemma~\ref{lem2.4}, upon summing over $c$ against the weight function $w$, we obtain the following result (the formula in \eqref{Sstar}) for $D(\phi;w, X)$.

\begin{lemma}\label{lemma explicit formula all d}  Assume that $\phi$ is an even Schwartz test function whose Fourier transform has compact support. Then we have
\begin{align*}
\begin{split}
D(\phi;w, X) =  &\frac {\widehat \phi(0)}{ \mathcal{L} W(X)}\underset{(c, 1+i) =1 }{\sum \nolimits^{*}} w\left( \frac {N(c)}X \right)  \log N(c) +\frac {\widehat\phi(0)}{\mathcal{L}} \left( \log\frac {32}{\pi^2}+ 2\frac{\Gamma'}{\Gamma} \left( \frac 12 \right) \right ) \\
&\hspace*{1cm}  -\frac 2{\mathcal{L}W(X)} \underset{(c, 1+i) =1 }{\sum \nolimits^{*}}  w\left( \frac {N(c)}X \right) \sum_{j \geq 1} S_j(\chi_{i(1+i)^5c}, \mathcal{L};\hat{\phi}) +\frac{2}{\mathcal{L}}\int\limits_0^\infty\frac{e^{-x/2}}{1-e^{-x}}\left(\widehat{\phi}(0)-\widehat{\phi}\left(\frac{x}{\mathcal{L}}\right)\right) \dif x.
\end{split}
\end{align*}
\end{lemma}

   Now, let $w(t)$ be an even, non-zero and non-negative Schwartz function as the theorems.  We define
\begin{align} \label{g}
 g(y)=\widetilde{w}\left(\sqrt{2 }y \right), \quad g_1(y)=\widetilde{g}\left( \sqrt{y} \right),
\end{align}
  where we recall that for any function $W$, the definition of $\widetilde{W}$ is given in \eqref{wtilde}. \newline

  Our next lemma establishes a relation between the Mellin transforms of $g$ and $g_2$. This is a generalization of \cite[(2.3)]{FPS1}.
\begin{lemma}
\label{Melliniden}
  For any $z \in \mc, z \neq 0, -1$, we have
\begin{equation*}
\zeta_K(z+1)\mathcal M g_1(z+1)=\zeta_K(-z) \mathcal M g(-z).
\end{equation*}
\end{lemma}
\begin{proof}
  Our proof of the lemma is motivated by Riemann's proof of the functional equation of the Riemann zeta function $\zeta(s)$ (see \cite[\S 8]{Da}). We first note that, for $\Re(z)>1$,
\begin{align*}
\zeta_K(z) \mathcal M g(z)=& \frac 14\sum_{\substack{ k \in \mathcal O_K \\ k \neq 0}}\int\limits^{\infty}_0g(t)\left( \frac {t}{N(k)} \right)^z \frac {\dif t}{t} = \frac 14\sum_{\substack{ k \in \mathcal O_K \\ k \neq 0}}\int\limits^{\infty}_0g(N(k) t)t^z \frac {\dif t}{t}  \\
&= -\frac 14\frac {g(0)}z+ \frac 14  \int\limits^{\infty}_1  \sum_{\substack{ k \in \mathcal O_K \\ k \neq 0}} g(N(k) t)  t^z \frac {\dif t}{t}+ \frac 14\int\limits^{1}_0 \sum_{\substack{ k \in \mathcal O_K }}g(N(k) t) t^z \frac {\dif t}{t}
\end{align*}
Now applying \eqref{tripoi} to the second of the above sums and another change of variables, we get
\[ \zeta_K(z) \mathcal M g(z) = -\frac 14\frac {g(0)}z+\frac 14\frac {\widetilde g(0)}{z-1}+  \frac 14\int\limits^{\infty}_1 \sum_{\substack{ k \in \mathcal O_K \\ k \neq 0}}g(N(k) t) t^z \frac {\dif t}{t}+ \frac 14\int\limits^{\infty}_1 \sum_{\substack{ k \in \mathcal O_K \\ k \neq 0}}\widetilde g(\sqrt{N(k) t}) t^{-z+1}  \frac {\dif t}{t}. \]
  Note that the last two integrals converge absolutely for all $z \in \mc$, by applying estimation \eqref{bounds} to both $g$ and $\widetilde g$. The last expression above thus gives an analytical extension of $\zeta_K(z) \mathcal M g(z)$ to all $z \in \mc, z \neq 0$, 1. \newline

  Similarly, we also deduce from \eqref{tripoi} that for $\Re(z)>0$,
\begin{align*}
\zeta_K(z+1) \mathcal M g_1(z+1)=& \frac 14\sum_{\substack{ k \in \mathcal O_K \\ k \neq 0}}\int\limits^{\infty}_0 \widetilde g \left( \sqrt{t} \right) \left( \frac {t}{N(k)} \right)^{z+1} \frac {\dif t}{t} \\
=& \frac 14\frac {g(0)}z-\frac 14\frac {\widetilde g(0)}{z+1}+\frac 14\int\limits^{\infty}_1 \sum_{\substack{ k \in \mathcal O_K \\ k \neq 0}}\widetilde g \left( \sqrt{N(k)t} \right) t^{z+1}   \frac {\dif t}{t}+\frac 14\int\limits^{\infty}_{1} \sum_{\substack{ k \in \mathcal O_K \\ k \neq 0}} g(N(k)t) t^{-z}   \frac {\dif t}{t}.
\end{align*}
  Once again by applying estimation \eqref{bounds} to both $g$ and $\widetilde g$, we see that the last two integrals above converge absolutely for all $z \in \mc$, so the last expression above thus gives an analytical extension of $\zeta_K(z+1) \mathcal M g_1(z+1)$ to all $z \in \mc, z \neq 0$, 1. Now, by comparing the above expressions for $\zeta_K(z) \mathcal M g(z)$ and $\zeta_K(z+1) \mathcal M g_1(z+1)$, we readily deduce the assertion of the lemma.
\end{proof}

\section{Analyzing sums over primes} \label{anaprime}

    We devote this section to the analysis of the sum over primes in \eqref{Sstar}.  We first separate the odd and the even prime powers, by writing
\begin{align}
\label{Sodd}
 S_{\text{odd}}= -\frac 2{\mathcal{L}W(X)} \underset{(c, 1+i) =1 }{\sum \nolimits^{*}} w\left( \frac {N(c)}X \right) \sum_{\substack{j \geq 1 \\ j \equiv 1 \pmod 2}} S_j(\chi_{i(1+i)^5c},\mathcal{L};\hat{\phi})
\end{align}
 and similarly for $S_{\text{even}}$.  Moreover,  it follows from Lemma \ref{lemma count of squarefree}, \eqref{mer} and \eqref{W} that
\begin{equation}
\label{Seven}
S_{\text{even}} = - \frac 2{\mathcal{L}}\sum_{\substack{\varpi \equiv 1 \bmod {(1+i)^3} \\ j\geq 1}} \frac{\log
   N(\varpi)}{N(\varpi)^{j}} \left(  1+\frac 1{N(\varpi)} \right)^{-1} \widehat \phi\left( \frac{2j \log
   N(\varpi)}{\mathcal{L}} \right) +O \left(X^{-3/4+\varepsilon }\right),
\end{equation}

\subsection{Estimation of $S_{\text{even}}$}

    We first expand $S_{\text{even}}$ into descending powers of $\LL$, we generalize \cite[Lemma 3.7]{FPS} to obtain the following result.
\begin{lemma}\label{lemma Seven}
Suppose that $\sigma=\sup \left(\text{supp } \widehat \phi \right)<\infty$. Then for any integer $M \geq 1$, we have the expansion
\begin{align}
\label{Sevenexpn}
  S_{\text{even}} =-\frac{\phi(0)}2 + \sum_{m=1}^M \frac{d_m \widehat \phi^{(m-1)}(0)}{\mathcal{L}^m} +O\left( \frac 1{\mathcal{L}^{M+1}} \right),
\end{align}
where the coefficients $d_m$ are real numbers that can be given explicitly.
\end{lemma}
\begin{proof}
  It suffices to show that the expansion given in \eqref{Sevenexpn} is valid if we ignore the $O \left(X^{-3/4+\varepsilon} \right)$ term in \eqref{Seven}. As $\sigma$ is finite, the sum in \eqref{Seven} is finite as we must have $N(\varpi)^{2j} \leq X^{\sigma}$. It follows that sum of the terms with $ j \geq 2$ can be expanded as follows:
\begin{align}
\label{eq: evencase}
\begin{split}
-\frac 2{\mathcal{L}} &\sum_{\substack{\varpi \equiv 1 \bmod {(1+i)^3} \\ j\geq 2 , \ N(\varpi)^{2j} \leq X^{\sigma}}} \frac{\log
   N(\varpi)}{N(\varpi)^{j}} \left(  1+\frac 1{N(\varpi)} \right)^{-1} \Bigg(\sum_{m=0}^M \frac{\widehat \phi^{(m)}(0)}{m!} \left(\frac{2j \log N(\varpi)}{\mathcal{L}} \right)^m+O\bigg( \bigg(\frac{ 2j\log N(\varpi)}{\mathcal{L}}\bigg)^{M+1} \bigg) \Bigg) \\
&= - \frac {2 }{\mathcal{L}} \sum_{m=0}^M \frac{\widehat \phi^{(m)}(0)}{m!\mathcal{L}^m}\sum_{\substack{\varpi \equiv 1 \bmod {(1+i)^3} \\ j\geq 2 , \ N(\varpi)^{2j} \leq X^{\sigma}}} \frac{\log N(\varpi)(2j \log N(\varpi))^m}{N(\varpi)^{j}} \left(  1+\frac 1{N(\varpi)}\right)^{-1} +O\left( \mathcal{L}^{-M-2} \right) \\
&=  - \frac {2 }{\mathcal{L}} \sum_{m=0}^M \frac{\widehat \phi^{(m)}(0)}{m!\mathcal{L}^m}\sum_{\substack{\varpi \equiv 1 \bmod {(1+i)^3} \\ j\geq 2 }} \frac{\log N(\varpi)(2j \log N(\varpi))^m}{N(\varpi)^{j}} \left(  1+\frac 1{N(\varpi)}\right)^{-1} +O\left( \mathcal{L}^{-M-2} \right),
\end{split}
\end{align}
  by noting the inner sum of the last expression above converges. \newline

 It remains to expand the terms with $j=1$. For this,  we first note that, using the Taylor expansion of $\hat{\phi}$ around the origin and rewriting $(1+N(\varpi)^{-1})^{-1}$ as a geometric series,
\begin{align}
\label{eq: oddcase}
\begin{split}
& -\frac 2{\mathcal{L}}\sum_{\substack{\varpi \equiv 1 \bmod {(1+i)^3}}} \frac{\log
   N(\varpi)}{N(\varpi)} \left(  1+\frac 1{N(\varpi)} \right)^{-1} \widehat \phi\left( \frac{2 \log
   N(\varpi)}{\mathcal{L}} \right)  \\
=&-\frac 2{\mathcal{L}}\sum_{\substack{\varpi \equiv 1 \bmod {(1+i)^3}}} \frac{\log
   N(\varpi)}{N(\varpi)} \left(  1+\frac 1{N(\varpi)} \right)^{-1} \Bigg(\sum_{m=0}^M \frac{\widehat \phi^{(m)}(0)}{m!} \left(\frac{2 \log N(\varpi)}{\mathcal{L}} \right)^m+O\bigg( \bigg(\frac{ \log N(\varpi)}{\mathcal{L}}\bigg)^{M+1} \bigg) \Bigg) \\
 =&-\frac 2{\mathcal{L}}\sum_{\substack{\varpi \equiv 1 \bmod {(1+i)^3}}} \frac{\log
   N(\varpi)}{N(\varpi)} \Bigg(\sum_{m=0}^M \frac{\widehat \phi^{(m)}(0)}{m!} \left(\frac{2 \log N(\varpi)}{\mathcal{L}} \right)^m+O\bigg( \bigg(\frac{ \log N(\varpi)}{\mathcal{L}}\bigg)^{M+1} \bigg) \Bigg)  \\
  & \hspace*{3cm} -\sum_{m=0}^M \frac{\widehat \phi^{(m)}(0)C_1(m)}{m!\mathcal{L}^{m+1}}+O\big(\mathcal{L}^{-M-1} \big) \\
 =&-\frac 2{\mathcal{L}}\sum_{\substack{\varpi \equiv 1 \bmod {(1+i)^3}}} \frac{\log
   N(\varpi)}{N(\varpi)}\widehat \phi\left( \frac{2 \log
   N(\varpi)}{\mathcal{L}} \right)  -\sum_{m=0}^M \frac{\widehat \phi^{(m)}(0)C_1(m)}{m!\mathcal{L}^{m+1}}+O_{K}\big(\mathcal{L}^{-M-1}\big)
\end{split}
\end{align}
   where
\begin{align*}
& C_1(m) =\sum_{\substack{\varpi \equiv 1 \bmod {(1+i)^3}}} \sum_{l\geq 1} (-1)^l  \frac{(2\log N(\varpi))^{m+1}}{N(\varpi)^{l+1}}
< \infty.
\end{align*}

Next, we note that
\begin{align} \label{grhbound}
E(t) := \sum_{\substack{ N(\varpi) \leq t \\ \varpi \equiv 1 \bmod {(1+i)^3}}} \log N(\varpi)- t \ll t^{1/2+\varepsilon}
\end{align}
from \eqref{lem2.7eq}. \newline

   We then apply partial summation to arrive at, with $E(t)$ defined in \eqref{grhbound},
\begin{align}
\label{eq: oddcasemain}
\begin{split}
& -\frac 2{\mathcal{L}}\sum_{\substack{\varpi \equiv 1 \bmod {(1+i)^3}}} \frac{\log
   N(\varpi)}{N(\varpi)}\widehat \phi\left( \frac{2 \log
   N(\varpi)}{\mathcal{L}} \right)= -\frac 2{\mathcal{L}}\int\limits_{1}^{\infty}  \frac 1t \widehat \phi\left( \frac{2 \log t}{\mathcal{L}} \right) \dif ( t+E(t)) \\
& = -\int\limits_{0} ^{\infty}   \widehat \phi\left(u\right)\, \dif u
+\frac 2{\mathcal{L}}\int\limits_{1}^{\infty} E(t) \frac{\dif}{\dif t} \left(\frac{1}{t} \widehat \phi\left( \frac{2 \log t}{\mathcal{L}} \right) \right) \, \dif t = -\frac 12 \phi(0)+\frac 2{\mathcal{L}}\int\limits_{1}^{\infty} E(t) \frac{\dif}{\dif t} \left( \frac{1}{t} \widehat \phi\left( \frac{2 \log t}{\mathcal{L}} \right) \right) \, \dif t .
\end{split}
\end{align}
  We can now expand the derivative in the last integrand in \eqref{eq: oddcasemain} into Taylor expansions involving powers of $2\log t/\mathcal{L}$ and use the corresponding series to calculate the last integral above. Note that the new integrals emerging from this process are all convergent because of the bound \eqref{grhbound}.  The assertion of the lemma now follows by combining \eqref{eq: evencase}, \eqref{eq: oddcase} and \eqref{eq: oddcasemain}.
\end{proof}

\subsection{Estimation of $S_{\text{odd}}$: Poisson summation}
\label{section Poisson 8d}

   Starting from this section, we shall concentrate on the estimation of $S_{\text{odd}}$. First note that the contribution from the terms with $j\geq 3$ in \eqref{Sodd} is $O(X^{-3/4+\varepsilon})$ by Lemma \ref{lemma count of squarefree}. It thus remains to treat the case for $j=1$. For this case,  we use the M\"obius function to detect the condition that $c$ is square-free to get
\begin{align*}
S_{\text{odd}} =& -\frac 2{\mathcal{L} W(X)  }\sum_{\substack {l \equiv 1 \mod (1+i)^3}}\mu_{[i]}(l)\sum_{\substack{ \varpi \equiv 1 \bmod {(1+i)^3} }} \frac {\log N(\varpi)}{\sqrt{N(\varpi)}} \hat{\phi} \left( \frac {\log N(
   \varpi)}{\log X} \right) \sum_{(c, 1+i)=1} \leg {i(1+i)cl^2}{\varpi} w \left( \frac {N(cl^2)}{X} \right)\\
  &  \hspace*{3cm} +O\big(X^{-3/4+\varepsilon}\big).
\end{align*}
We divide the sum over $l$ above into two parts, one over $l \leq Z$ and the other over $l>Z$, with $Z$ to be chosen optimally later. Note that if $c$ is odd, then $i(1+i)^5cl^2$ is never a square. Similar to the treatment of $S_R(X,Y; \hat{\phi}, \Phi)$ in \cite[Section 3.3]{G&Zhao4} except that we use Lemma \ref{lem2.7} here instead of \cite[Lemma 2.5]{G&Zhao4}, we get that the terms with $l> Z$ are
$$ \ll X^{\varepsilon}(\log Z)^3Z^{-1}.$$
   For the terms with $l \leq Z$, we apply the Poisson summation \eqref{quadpoi} given in Lemma \ref{Poissonsum} and argue as in \cite{G&Zhao4} (the treatment here is essentially the same as that of $S_M(X,Y; \hat{\phi}, \Phi)$ in \cite[Section 3.2]{G&Zhao4}) to arrive at the following lemma.

\begin{lemma}
 Suppose that GRH is true. We have for any $Z \geq 1$ and any $\epsilon>0$,
\begin{equation} \label{equation thing to prove in lemma Poisson}
\begin{split}
S_{\text{odd}} =& -\frac X{\mathcal{L}  W(X)}\sum_{\substack {N(l) \leq Z \\ l \equiv 1 \bmod {(1+i)^3}}} \frac {\mu_{[i]}(l)}{N(l^2)} \sum_{k \in
   \mz[i]}(-1)^{N(k)} \sum_{\varpi \equiv 1 \bmod {(1+i)^3}} \frac {\log N(\varpi)}{N(\varpi)}\leg {kl^2}{\varpi}\hat{\phi}\left( \frac {\log N(
   \varpi)}{\log X} \right) \widetilde{w}\left(\sqrt{\frac {N(k)X}{2N(l^2\varpi)}}\right) \\
&\hspace*{3cm} +O\left( X^{-3/4+\varepsilon}+X^{\varepsilon}(\log Z)^3Z^{-1} \right).
\end{split}
\end{equation}
\end{lemma}

   We now generalize \cite[Lemma 2.7]{FPS} to further analyze the sums in \eqref{equation thing to prove in lemma Poisson}, obtaining the following result.
\begin{lemma}
\label{lemma:small s}
  Suppose that GRH is true and that $\sigma=\text{sup}(\text{supp } \widehat \phi)<\infty$. Then we have for any $1 \leq Z \leq X^{2}$ and any $\varepsilon>0$
\begin{align}
\label{Soddformula}
\begin{split}
S_{\text{\emph{odd}}} =& \frac X{ W(X)  } \sum_{\substack{N(l) \leq Z \\ l \equiv 1 \bmod {(1+i)^3}}} \frac {\mu_{[i]}(l)}{N(l^2)} \Big (\frac 12 I_{(1+i)l}(X)-I_l(X) \Big ) \\
& \hspace*{3cm} +O\left(X^{-3/4+\varepsilon}+X^{\varepsilon}(\log Z)^3Z^{-1}+ZX^{\sigma/2-1+\varepsilon}+X^{-1/2+\varepsilon}Z^{\varepsilon} \right),
\end{split}
\end{align}
  where
\begin{equation}
\label{Il2}
  I_l(X) =  \int\limits_{0}^{\infty} \widehat \phi( u )   \sum_{\substack {k \in \mz[i] \\ k \neq 0}} \widetilde{w}\left(2N(k)\sqrt{\frac {X^{1-u}}{2N(l^2)}}\right) \dif u.
  \end{equation}
\end{lemma}
\begin{proof}
   Note that as in \cite[Section 3.4]{G&Zhao4} that the inner sum in \eqref{equation thing to prove in lemma Poisson} corresponding to $k=0$ is zero. It also follows from the treatment of \cite[Section 3.5]{G&Zhao4} by setting $U=1$ and dividing the estimation obtained in \cite[(3.9)]{G&Zhao4} by $X$ (since our definition of $S_{\text{\emph{odd}}}$ differs from $S_M(X,Y; \hat{\phi}, \Phi)$ defined in \cite{G&Zhao4} by an extra factor $W^{-1}(X)$) that the contribution of $k \neq \square$ ($k$ is not a square) to the expression for $S_{\text{\emph{odd}}}$ given in  \eqref{equation thing to prove in lemma Poisson} is
\begin{align*}
   \ll ZX^{\sigma/2-1+\varepsilon}.
\end{align*}

      We are left to consider the contribution from $k=\square$ ($k$ is a square), $k \neq 0$ to the expression for $S_{\text{\emph{odd}}}$ given in  \eqref{equation thing to prove in lemma Poisson}. For this, we make a change of variables $k \mapsto k^2$ while noting that $k^2_1=k^2_2$ if and only if $k_1 = \pm k_2$ and deduce that
\begin{align*}
S_{\text{odd}} = -\frac X{2 \mathcal{L}  W(X)  } & \sum_{\substack{N(l) \leq Z \\ l \equiv 1 \bmod {(1+i)^3}}} \frac {\mu_{[i]}(l)}{N(l^2)} \sum_{\substack {(\varpi , l )=1 \\ \varpi \equiv 1 \bmod {(1+i)^3} }} \frac {\log N(\varpi)}{N(\varpi)}\hat{\phi} \left( \frac {\log N(
   \varpi)}{\log X} \right)  \sum_{\substack {k \in
   \mz[i] , k \neq 0 \\ (k, \varpi)=1}}(-1)^{N(k)}\widetilde{w}\left(N(k)\sqrt{\frac {X}{2N(l^2\varpi)}}\right) \\
   &+O\left(X^{-3/4+\varepsilon}+X^{\varepsilon}(\log Z)^3Z^{-1}+ZX^{\sigma/2-1+\varepsilon} \right).
\end{align*}

    In view of the rapid decay property of $\widetilde{w}$ implied by \eqref{bounds},  we now remove the condition that $(\varpi, l)=1$ at the cost of an error
\begin{align*}
 \ll & \frac 1{\mathcal{L}}\sum_{\substack{N(l) \leq Z \\ l \equiv 1 \bmod {(1+i)^3}}} \frac {1}{N(l^2)} \sum_{\substack {\varpi | l \\ \varpi \equiv 1 \bmod {(1+i)^3} }} \frac {\log N(\varpi)}{N(\varpi)}\sum_{\substack {k \in
   \mz[i] ,\ k \neq 0 \\ N(k) \leq \sqrt{2N(l^2\varpi)/X}}}1 \\
 \ll & X^{-1/2}\sum_{\substack{N(l) \leq Z \\ l \equiv 1 \bmod {(1+i)^3}}} \frac {1}{N(l)} \sum_{\substack {\varpi | l \\ \varpi \equiv 1 \bmod {(1+i)^3} }} \frac {\log N(\varpi)}{\sqrt{N(\varpi)}} \ll X^{-1/2}Z^{\varepsilon},
\end{align*}
  where we use the well-known fact that for $N(l) \geq 3$,  the number, $\omega(l)$, of distinct primes in $\mz[i]$ dividing $l$  can be bounded as
\begin{align*}
   \omega(l) \ll \frac {\log N(l)}{\log \log N(l)},
\end{align*}

   One can show similarly that removing the condition $(k,\varpi)=1$ introduces an error of size $\ll X^{-1/2}Z^{\varepsilon}$.  We thus derive, using \eqref{lem2.7eq}, that
\begin{align*}
\begin{split}
S_{\text{odd}} =& -\frac X{2 \mathcal{L} W(X)}\sum_{\substack{N(l) \leq Z \\ l \equiv 1 \bmod {(1+i)^3}}} \frac {\mu_{[i]}(l)}{N(l^2)} \sum_{\substack {\varpi \equiv 1 \bmod {(1+i)^3} }} \frac {\log N(\varpi)}{N(\varpi)}\hat{\phi} \left( \frac {\log N(
   \varpi)}{\log X} \right) \\
   & \hspace*{3cm} \times \sum_{\substack {k \in
   \mz[i] \\ k \neq 0}}(-1)^{N(k)}\widetilde{w}\left(N(k)\sqrt{\frac {X}{2N(l^2\varpi)}}\right)\\
   &+O\left(X^{-3/4+\varepsilon}+X^{\varepsilon}(\log Z)^3Z^{-1}+ZX^{\sigma/2-1+\varepsilon}+X^{-1/2+\varepsilon}Z^{\varepsilon} \right) \\
   =& -\frac {X}{2 \mathcal{L}  W(X)  }\sum_{\substack{N(l) \leq Z \\ l \equiv 1 \bmod {(1+i)^3}}} \frac {\mu_{[i]}(l)}{N(l^2)} \sum_{\substack {k \in
   \mz[i] \\ k \neq 0}}(-1)^{N(k)} \\
   & \hspace*{3cm} \times \int\limits^{\infty}_1\frac 1y \hat{\phi} \left( \frac {\log y}{\mathcal{L}} \right)\widetilde{w}\left(N(k)\sqrt{\frac {X}{2N(l^2)y}}\right)\dif \big(y+ O\big( y^{1/2} \log^{2} (2y)\big ) \big ) \\
   & +O\left(X^{-3/4+\varepsilon}+X^{\varepsilon}(\log Z)^3Z^{-1}+ZX^{\sigma/2-1+\varepsilon}+X^{-1/2+\varepsilon}Z^{\varepsilon} \right) \\
   =& S_1+S_2+O\left(X^{-3/4+\varepsilon}+X^{\varepsilon}(\log Z)^3Z^{-1}+ZX^{\sigma/2-1+\varepsilon}+X^{-1/2+\varepsilon}Z^{\varepsilon} \right),
\end{split}
\end{align*}
   where
\[ S_1= -\frac{X}{2 \mathcal{L}W(X)} \sum_{\substack{N(l) \leq Z \\ l \equiv 1 \bmod {(1+i)^3}}} \frac {\mu_{[i]}(l)}{N(l^2)}  \sum_{\substack {k \in
   \mz[i] \\ k \neq 0}}(-1)^{N(k)}\int\limits^{\infty}_1\frac 1y  \hat{\phi} \left( \frac {\log y}{\mathcal{L}} \right)\widetilde{w}\left(N(k)\sqrt{\frac {X}{2N(l^2)y}}\right)\dif y, \]
   and
\[ S_2  = -\frac{X}{2 \mathcal{L}W(X)}\sum_{\substack{N(l) \leq Z \\ l \equiv 1 \bmod {(1+i)^3}}} \frac {\mu_{[i]}(l)}{N(l^2)} \sum_{\substack {k \in
   \mz[i] \\ k \neq 0}}(-1)^{N(k)}\int\limits^{\infty}_1 \frac 1y \hat{\phi}\left( \frac {\log y}{\mathcal{L}} \right)\widetilde{w}\left(N(k)\sqrt{\frac {X}{2N(l^2)y}}\right)\dif \big( O\big( y^{1/2} \log^{2} (2y)\big ) \big ). \]

    Note that
\begin{align*}
S_2  \ll& \frac 1{\mathcal{L}}  \sum_{\substack{N(l) \leq Z \\ l \equiv 1 \bmod {(1+i)^3}}} \frac {1}{N(l^2)}  \sum_{\substack {k \in
   \mz[i] \\ k \neq 0}}\int\limits^{\infty}_1  y^{1/2} \log^{2} (2y) \frac{\dif}{\dif y} \left(\frac 1y \hat{\phi} \left( \frac {\log y}{\mathcal{L}} \right)\widetilde{w}\left(N(k)\sqrt{\frac {X}{2N(l^2)y}}\right) \right ) \dif y \ll S_{2,1}+S_{2,2},
\end{align*}
  where
\[ S_{2,1} = \frac 1{\mathcal{L}} \sum_{\substack{N(l) \leq Z \\ l \equiv 1 \bmod {(1+i)^3}}} \frac {1}{N(l^2)}  \sum_{\substack {k \in
   \mz[i] \\ k \neq 0}}\Bigg | \widetilde{w}\left(N(k)\sqrt{\frac {X}{2N(l^2)}}\right)\Bigg | \]
and
\begin{align*}
 S_{2,2} =  \frac 1{\mathcal{L}} & \sum_{\substack{N(l) \leq Z \\ l \equiv 1 \bmod {(1+i)^3}}}  \frac {1}{N(l^2)} \int\limits^{\infty}_1 \sum_{\substack {k \in
   \mz[i] \\ k \neq 0}} y^{1/2} \log^{2} (2y)\Bigg( \Bigg |\Bigg( \frac 1{y^2\mathcal{L}}\hat{\phi}' \left( \frac {\log y}{\mathcal{L}} \right)\widetilde{w}\left(N(k)\sqrt{\frac {X}{2N(l^2)y}}\right)\Bigg | \\
   & +\Bigg |\frac 1{y^2} \hat{\phi} \left( \frac {\log y}{\mathcal{L}} \right)\widetilde{w}\left(N(k)\sqrt{\frac {X}{2N(l^2)y}}\right)\Bigg | +\Bigg | \hat{\phi} \left( \frac {\log y}{\mathcal{L}} \right)N(k)\sqrt{\frac {X}{8N(l^2)}}y^{-5/2}\widetilde{w}'\left(N(k)\sqrt{\frac {X}{2N(l^2)y}}\right)\Bigg )\Bigg | \Bigg ) \dif y .
\end{align*}

    Using \eqref{bounds}, we deduce that
\[ \sum_{\substack {k \in
   \mz[i] \\ k \neq 0}}\Bigg |\widetilde{w}\left(N(k)\sqrt{\frac {X}{2N(l^2)y}}\right)\Bigg | \ll \sqrt {\frac {N(l^2)y}{X}} \quad \mbox{and} \quad \sum_{\substack {k \in
   \mz[i] \\ k \neq 0}}\Bigg | N(k)\widetilde{w}'\left(N(k)\sqrt{\frac {X}{2N(l^2)y}}\right)\Bigg |   \ll \frac {N(l^2)y}{X}. \]

 Thus, it follows that
\begin{align*}
 S_{2,2}  \ll&  \frac 1{\mathcal{L}} \sum_{\substack{N(l) \leq Z \\ l \equiv 1 \bmod {(1+i)^3}}} \frac {1}{N(l^2)}\int\limits^{\infty}_1  y^{1/2} \log^{2} (2y)  \\
  & \hspace*{1cm} \times \left( \frac 1{y^2\mathcal{L}}\Bigg |\hat{\phi} '\left( \frac {\log y}{\mathcal{L}} \right)\Bigg |\sqrt {\frac {N(l^2)y}{X}}+\frac 1{y^2}\Bigg | \hat{\phi}\left( \frac {\log y}{\mathcal{L}} \right)\sqrt {\frac {N(l^2)y}{X}}\Bigg |+ \Bigg |\hat{\phi} \left( \frac {\log y}{\mathcal{L}} \right)\Bigg | \sqrt{\frac {X}{8N(l^2)}}y^{-5/2}\frac {N(l^2)y}{X}\right )\dif y  \\
  \ll & X^{-1/2}Z^{\varepsilon}\int\limits^{\infty}_1   \log^{2} (2y)\left ( \Bigg | \hat{\phi}'\left( \frac {\log y}{\mathcal{L}} \right)\Bigg |+ \Bigg |\hat{\phi} \left( \frac {\log y}{\mathcal{L}} \right)\Bigg | \right ) \dif \left ( \frac {\log y}{\mathcal{L}} \right ) \ll X^{-1/2+\varepsilon}Z^{\varepsilon},
\end{align*}
   where the last estimation above follows by using a change of variable  $u = \log y/\mathcal{L}$ to evaluate the proceeding integral and noting that the integrand has compact support. Similarly, we have that $S_{2,1}  \ll  X^{-1/2+\varepsilon}Z^{\varepsilon}$ so that
\begin{align*}
 S_{2}  \ll&  X^{-1/2+\varepsilon}Z^{\varepsilon}.
\end{align*}

   It follows from this that $S_2$ contributes to the error term in \eqref{Soddformula}. It remains to evaluate $S_1$ and applying exactly the same change of variable as above now leads to
\begin{align}
\label{S1}
 S_{1} = -\frac X{ 2W(X)  } \sum_{\substack{N(l) \leq Z \\ l \equiv 1 \bmod {(1+i)^3}}} \frac {\mu_{[i]}(l)}{N(l^2)}
\int\limits_{0}^{\infty} \widehat \phi( u )   \sum_{\substack {k \in
   \mz[i] , k \neq 0}}(-1)^{N(k)} \widetilde{w}\left(N(k)\sqrt{\frac {X^{1-u}}{2N(l^2)}}\right) \dif u.
\end{align}

   We note that, for any $l \in \mathcal O_K$, we have
\begin{align*}
  \sum_{\substack {k \in
   \mz[i] \\ k \neq 0}}(-1)^{N(k)} \widetilde{w}\left(N(k)\sqrt{\frac {X^{1-u}}{2N(l^2)}}\right) =& \sum_{\substack {k \in
   \mz[i] \\ k \neq 0 \\ 1+i|k}}\widetilde{w}\left(N(k)\sqrt{\frac {X^{1-u}}{2N(l^2)}}\right)-\sum_{\substack {k \in
   \mz[i] \\ k \neq 0 \\ (1+i,k)=1}}\widetilde{w}\left(N(k)\sqrt{\frac {X^{1-u}}{2N(l^2)}}\right) \\
=& 2\sum_{\substack {k \in
   \mz[i] \\ k \neq 0 \\ 1+i|k}}\widetilde{w}\left(N(k)\sqrt{\frac {X^{1-u}}{2N(l^2)}}\right)-\sum_{\substack {k \in
   \mz[i] \\ k \neq 0}}\widetilde{w}\left(N(k)\sqrt{\frac {X^{1-u}}{2N(l^2)}}\right) \\
= & 2\sum_{\substack {k \in
   \mz[i] \\ k \neq 0 }}\widetilde{w}\left(2 N(k)\sqrt{\frac {X^{1-u}}{2N(l^2)}}\right)-\sum_{\substack {k \in
   \mz[i] \\ k \neq 0}}\widetilde{w}\left(N(k)\sqrt{\frac {X^{1-u}}{2N(l^2)}}\right).
\end{align*}

  Applying the above in \eqref{S1}, we derive that
\begin{align*}
S_{1} = \frac X{  W(X)  } \sum_{\substack{N(l) \leq Z \\ l \equiv 1 \bmod {(1+i)^3}}} \frac {\mu_{[i]}(l)}{N(l^2)} \Big (\frac 12 I_{(1+i)l}(X)-I_l(X) \Big ).
\end{align*}

  The above gives precisely the main term in \eqref{Soddformula} for $S_{\text{odd}}$ and this completes the proof.
\end{proof}

\subsection{Estimation of $S_{\text{odd}}$: small support}

  In this section, we apply Lemma \ref{lemma:small s} to show that there is no new lower order terms in powers of $\mathcal{L}^{-1}$ when $\sigma=\text{sup}(\text{supp } \widehat \phi)< 1$. We generalize \cite[Proposition 3.1]{FPS} and state our result in the following
\begin{proposition}
\label{proposition GRH small support}
 Suppose that GRH is true and that $\sigma=\sup(\text{supp } \widehat \phi)<1$.  Then we have for any $\varepsilon>0$,
$$  S_{\text{odd}}  \ll X^{\sigma/4 - 1/2+\varepsilon} + X^{3\sigma/4-3/4+\varepsilon}. $$
\end{proposition}

\begin{proof}
 We let
\begin{align}
\label{Phi}
  \Phi(X)= \sum_{\substack{N(l) \leq X \\ l \equiv 1 \bmod {(1+i)^3} }} \frac{\mu_{[i]}(l)}{N(l)^2}=  \frac 4{3\zeta_K(2)} +O\big(X^{-3/2 + \varepsilon}\big).
\end{align}
   where the last equality above follows from the observation that under GRH we have
   $$\sum_{\substack{N(l) \leq X \\ (l, 1+i)=1 }} \mu_{[i]}(l) \ll X^{1/2+\varepsilon}.$$

We then deduce that for $0\leq u \leq 1$,
\begin{align*}
 \sum_{\substack{N(l) \leq Z \\ l \equiv 1 \bmod {(1+i)^3}}} & \frac {\mu_{[i]}(l)}{N(l^2)}\widetilde{w}\left(N(k)\sqrt{\frac {X^{1-u}}{2N(l^2)}}\right)  = \int\limits_{0^+}^Z \widetilde{w}\left(\frac {N(k)}{t}\sqrt{\frac {X^{1-u}}{2}}\right) \dif \Phi(t) \\
 & =\int\limits_{0^+}^Z \widetilde{w}\left(\frac {N(k)}{t}\sqrt{\frac {X^{1-u}}{2}}\right) \dif \Big(\frac 4{3\zeta_K(2)}+O\big(t^{-3/2 + \epsilon}\big)\Big ) \\
 & = \widetilde{w}\left(\frac {N(k)}{Z}\sqrt{\frac {X^{1-u}}{2}}\right) O\big(Z^{-3/2 + \varepsilon}\big) + N(k)\sqrt{\frac {X^{1-u}}{2}}\int\limits_{0^+}^Z \widetilde{w}'\left(\frac {N(k)}{t}\sqrt{\frac {X^{1-u}}{2}}\right) O \big(t^{-3/2 + \varepsilon}\big)\frac{ \dif t}{t^2} \\
& \ll Z^{-3/2 +\varepsilon} \bigg|\widetilde{w}\left(\frac {N(k)}{Z}\sqrt{\frac {X^{1-u}}{2}}\right)\bigg|+N(k)X^{(1-u)/2}\int\limits_{0^+}^Z  \bigg|\widetilde{w}'\left(\frac {N(k)}{t}\sqrt{\frac {X^{1-u}}{2}}\right) \bigg| \frac{\dif t}{t^{7/2-\varepsilon}}.
\end{align*}
Note that the part of the last integral for $ t \in (0,X^{(1-u)/2-\varepsilon}]$ is $O\big((N(k)X)^{-A}\big)$ for any $A\geq 1$, by the rapid decay of $\widetilde w'$. Summing over $k$ and integrating over $u$, we obtain that
\begin{align*}
\sum_{\substack{N(l) \leq Z \\ l \equiv 1 \bmod {(1+i)^3}}} \frac {\mu_{[i]}(l)}{N(l^2)}I_l(X) & \ll \int\limits_{0}^{\infty} \big|\widehat \phi( u )\big|  \sum_{\substack{k \in \mz[i] \\ k \neq 0}}  \Bigg(Z^{-3/2 +\varepsilon} \bigg|\widetilde{w}\left(\frac {N(k)}{Z}\sqrt{\frac {X^{1-u}}{2}}\right)\bigg|\\
&\hspace{2cm}+N(k)X^{(1-u)/2}\int\limits_{0^+}^Z  \bigg|\widetilde{w}'\left(\frac {N(k)}{t}\sqrt{\frac {X^{1-u}}{2}}\right) \bigg| \frac{\dif t}{t^{7/2-\varepsilon}} \Bigg)\,\dif u +X^{-1} \\
& \ll \int\limits_{0}^{\infty} \big|\widehat \phi( u )\big|  \bigg(Z^{-3/2 +\varepsilon} \frac{Z}{X^{(1-u)/2}}+\frac 1{X^{(1-u)/2}}\int\limits_{X^{(1-u)/2-\varepsilon}}^Z   \frac{\dif t}{t^{3/2-\varepsilon}} \bigg)\, \dif u + X^{-1}\\
& \ll \frac{X^{(\sigma-1)/2}}{Z^{1/2 -\varepsilon}} +  X^{3\sigma/4 -3/4+\varepsilon}.
\end{align*}
 Similarly, we have
\begin{align*}
\sum_{\substack{N(l) \leq Z \\ l \equiv 1 \bmod {(1+i)^3}}} \frac {\mu_{[i]}(l)}{N(l^2)}I_{(1+i)l}(X) \ll \frac{X^{(\sigma-1)/2}}{Z^{1/2 -\varepsilon}} +  X^{3\sigma/4 -3/4+\varepsilon}.
\end{align*}

  Hence, it follows from Lemma \ref{lemma:small s} that for $Z\leq X^2$,
$$ S_{\text{odd}} \ll \frac{X^{(\sigma-1)/2}}{Z^{1/2 -\varepsilon}}+X^{3\sigma/4 -3/4+\varepsilon} +X^{-3/4+\varepsilon}+X^{\varepsilon}Z^{-1}+ZX^{\sigma/2-1+\varepsilon}+X^{-1/2+\varepsilon}Z^{\varepsilon}.$$
The result follows by taking $Z=X^{1/2-\sigma/4}$.
\end{proof}

\subsection{Estimation of $S_{\text{odd}}$: extended support}
\label{section extended support}

   In this section, we analyze the lower order terms of $S_{\text{odd}}$ when $\sigma=\sup(\text{supp } \widehat \phi) \geq 1$.
\begin{lemma}
\label{lemma I_s(X)}
Suppose that $\sigma=\text{sup}(\text{supp } \widehat \phi)<\infty$. Then concerning the function $I_l(X)$ defined in \eqref{Il2}, we have
\begin{equation} \label{equation estimate I_s(X)}
\begin{split}
I_l(X)=& -\widetilde w (0)\int\limits_{1}^{\infty} \widehat \phi(u)\,du +\frac {\widetilde{g}(0)}{\LL} \int\limits^{\infty}_0\phi( 1+\tau/\LL ) e^{\tau/2}N(l) \, \dif \tau \\
 & +\frac 1{\LL} \int\limits_{0}^{\infty} \widehat \phi( 1+ \tau /\LL )  \bigg(  e^{\tau/2}N(l)\sum_{\substack{ j \in
   \mz[i] \\ j \neq 0}}\widetilde{g}\left(\sqrt {N(j) e^{\tau/2}N(l)}\right) \, \dif \tau  \\
   & +\frac 1{\LL}   \int\limits_0^{\infty} \widehat \phi( 1-\tau/\LL )   \sum_{\substack {k \in \mz[i] \\ k \neq 0}}g\left(N(k)\sqrt{\frac {e^{\tau}}{N(l^2)}}\right) \dif \tau+O\big(N(l)X^{-1/2}\big),
   \end{split}
   \end{equation}
   and
   \begin{equation} \label{equation estimate I_(1+i)s(X)}
   \begin{split}
I_{(1+i)l}(X)=& -\widetilde w (0)\int\limits_{1}^{\infty} \widehat \phi(u)\,du +\frac {2\widetilde{g}(0)}{\LL} \int\limits^{\infty}_0\phi( 1+ \tau/\LL ) e^{\tau/2}N(l) \, \dif \tau \\
 & +\frac 1{\LL} \int\limits_{0}^{\infty} \widehat \phi( 1+ \tau/\LL )  \bigg(  2e^{\tau/2}N(l)\sum_{\substack{ j \in
   \mz[i] \\ j \neq 0}}\widetilde{g}\left(2\sqrt {N(j) e^{\tau/2}N(l)}\right) \, \dif \tau  \\
   & +\frac 1{\LL}   \int\limits_0^{\infty} \widehat \phi( 1- \tau/\LL )   \sum_{\substack {k \in \mz[i] \\ k \neq 0}}g\left(\frac {N(k)}{2}\sqrt{\frac {e^{\tau}}{N(l^2)}}\right) \dif \tau,
\end{split}
\end{equation}
  where  $g(y)$ is given as in \eqref{g}.
\end{lemma}
\begin{proof}
   We first extend the integral in \eqref{Il2} to $\mr$ and make the substitution $\tau = \mathcal{L}(u-1)$ to obtain that
\begin{align}
\label{Iint}
 I_l(X)=\frac 1{\LL}   \int\limits_{-\infty}^{\infty} \widehat \phi \left( 1+\tau/\LL \right)   \sum_{\substack {k \in
   \mz[i] \\ k \neq 0}} \widetilde{w}\left(2N(k)\sqrt{\frac {e^{-\tau}}{2N(l^2)}}\right) \dif \tau +O(N(l)X^{-1/2}),
\end{align}
    since for $u \leq 0$, we have
\begin{align*}
  \int\limits^{0}_{-\infty} \widehat \phi( u )   \sum_{\substack {k \in
   \mz[i] \\ k \neq 0}}\widetilde{w}\left(2N(k)\sqrt{\frac {X^{1-u}}{2N(l^2)}}\right) \dif u \ll \int\limits^{0}_{-\infty} \widehat \phi( u )  \sqrt{\frac {N(l^2)}{X^{1-u}}} \dif u \ll N(l)X^{-1/2}.
\end{align*}

We then break the integral in \eqref{Iint} into integrals over $(-\infty,0]$ and $[0,\infty)$ and denote them respectively by $I_l^-(X)$ and $I_l^+(X)$.  By applying the Poisson summation formula \eqref{tripoi} given in Lemma \ref{Poissonsum}, we see that
\begin{align*}
I_l^+(X)&= \frac 1{\LL}\int\limits_{0}^{\infty} \widehat \phi( 1+\tau/\LL  )  \bigg(-\widetilde{w}(0)  + \sum_{k \in \mz[i]} g\left(N(k)\sqrt{\frac {e^{-\tau}}{N(l^2)}}\right) \bigg)\, \dif \tau \\
&=  \frac 1{\LL} \int\limits_{0}^{\infty} \widehat \phi( 1+ \tau/\LL )  \bigg(-\widetilde{w}(0)  +  e^{\tau/2}N(l)\sum_{\substack{ j \in
   \mz[i] }}\widetilde{g}\left(\sqrt{N(j) e^{\tau/2}N(l)}\right) \, \dif \tau \\
&=  \frac 1{\LL} \int\limits_{0}^{\infty} \widehat \phi( 1+ \tau/\LL )  \bigg(-\widetilde{w}(0)  +  e^{\tau/2}N(l)\widetilde{g}(0)+ e^{\tau/2}N(l)\sum_{\substack{ j \in
   \mz[i] \\ j \neq 0}}\widetilde{g}\left(\sqrt{N(j) e^{\tau/2}N(l)}\right) \, \dif \tau.
\end{align*}

Moreover, substituting $\tau$ by $-\tau$ in $I_l^-(X)$, we obtain that
$$I^-_l(X)=\frac 1{\LL}   \int\limits_0^{\infty} \widehat \phi( 1- \tau/\LL )   \sum_{\substack {k \in
   \mz[i] \\ k \neq 0}}g \left(N(k)\sqrt{\frac {e^{\tau}}{N(l^2)}}\right) \dif \tau. $$

  Combining the above expressions for $I_l^-(X)$ and $I_l^+(X)$, we readily derive the expression for $I_l(X)$ in \eqref{equation estimate I_s(X)}. The expression for $I_{(1+i)l}(X)$ in \eqref{equation estimate I_(1+i)s(X)} can be similarly obtained via the expression of $I_{l}(X)$ with the function $g(y)$ being replaced by $g(y/2)$ and this completes the proof.
\end{proof}

 We define the functions
\[  h_1(x) = \frac{3\zeta_K(2)}{\pi \widehat w(0)}\sum_{\substack{l \equiv 1 \bmod {(1+i)^3}}}\frac {\mu_{[i]}(l)}{N(l)}\Big ( \widetilde{g}\left(\sqrt{2 N(l)x } \right)- \widetilde{g}\left(\sqrt{ N(l)x } \right) \Big ) , \]
  and
\[  h_2(x)= \frac{3\zeta_K(2)}{\pi \widehat w(0)}\sum_{\substack{ l \equiv 1 \bmod {(1+i)^3}}}\frac {\mu_{[i]}(l)}{N(l^2)}\Big (\frac 12 g\left(\frac {x}{2N(l)}\right)- g\left(\frac {x}{N(l)}\right)\Big ). \]

It is easy to see that $h_1(x)$ and $h_2(x)$ are smooth on $(0, \infty)$ and $[0, \infty)$, respectively. Moreover, we have the bounds $h_1(x)\ll x^{-A}$ for any $A\geq 1$ and $h_2(x)\ll x^{-3/2+\varepsilon}$ for any $\varepsilon >0$ under GRH. We point out here that the above notations as well as their bounds are inspired by the corresponding notations introduced on page 1206 of \cite{FPS}. \newline

  We now apply Lemma \ref{lemma I_s(X)} to derive the following generalization of \cite[Corollary 3.4]{FPS}.
\begin{lemma}
\label{LemSodd}
Suppose that GRH is true.  Then we have for $\sigma<2$,
\begin{align*}
 S_{\text{odd}}  =   \int\limits_{1}^{\infty} \widehat \phi(u)\, \dif u+ J(X)+O\big(X^{\sigma/6-1/3+\varepsilon}\big),
\end{align*}
   where
\begin{align}
\label{corollary Sodd in terms of J}
 J(X)  =    \frac 1{\LL} \int\limits_{0}^{\infty} \bigg( \widehat \phi( 1+ \tau/\LL  )   e^{\tau/2} \sum_{\substack{ j \in
   \mz[i] \\  j \neq 0}} h_1\big(N(j)e^{\tau/2}\big)  +\widehat \phi( 1- \tau/\LL  )   \sum_{\substack {k \in
   \mz[i] \\ k \neq 0}} h_2\big(N(k)e^{\tau/2}\big)\bigg)\, \dif \tau,
\end{align}
\end{lemma}
\begin{proof}
First note that by following the arguments that lead to estimation \cite[(2.13)]{G&Zhao4} there, we have
\begin{align}
\label{w0}
 \widetilde{w}(0) =   \frac {\pi }{2}\widehat{w}(0).
\end{align}

   It follows from \eqref{W}, \eqref{Phi} and \eqref{w0} that
\begin{align}
\label{phihat}
   \frac{X}{W(X)} \sum_{\substack{N(l) \leq Z  \\ l \equiv 1 \bmod {(1+i)^3}}}\frac {\mu_{[i]}(l)}{N(l^2)} \widetilde{w}(0) \int\limits_{1}^{\infty} \widehat \phi(u)\, \dif u = 2 \int\limits_{1}^{\infty} \widehat \phi(u)\, \dif u +O(Z^{-3/2+\varepsilon}).
\end{align}

 We now combine Lemma~\ref{lemma:small s}, Lemma~\ref{lemma I_s(X)} and \eqref{phihat} together to get that if $Z \leq X^2$, then
\begin{align}
\label{Soddestim}
\begin{split}
S_{\text{odd}} = \frac{X}{W(X)} \sum_{\substack{N(l) \leq Z  \\ l \equiv 1 \bmod {(1+i)^3}}} & \frac {\mu_{[i]}(l)}{N(l^2)} \Bigg(  \int\limits_{0}^{\infty}  \widehat  \phi \left( 1+\tau/\mathcal{L} \right)\bigg(  e^{\tau/2}N(l)\sum_{\substack{ j \in
   \mz[i] \\ j \neq 0}}\left ( \widetilde{g}\left(2\sqrt {N(j) e^{\tau/2}N(l)}\right )-\widetilde{g}\left(\sqrt {N(j) e^{\tau/2}N(l)}\right)\right )  \, \dif \tau\\
& +\int\limits_{0}^{\infty} \widehat \phi \left( 1-\tau/\mathcal{L} \right)    \sum_{\substack {k \in
   \mz[i] \\ k \neq 0}}\left ( \frac 12 g\left(\frac {N(k)}{2}\sqrt{\frac {e^{\tau}}{N(l^2)}}\right)-g\left(N(k)\sqrt{\frac {e^{\tau}}{N(l^2)}}\right)\right )  \dif \tau  \Bigg )  \\
& +\int\limits_{1}^{\infty} \widehat \phi(u)\, \dif u +O\left(X^{-3/4+\varepsilon}+X^{\varepsilon}Z^{-1}+ZX^{\sigma/2-1+\varepsilon}+X^{-1/2+\varepsilon}Z^{\varepsilon} \right).
\end{split}
\end{align}

   Note that for the first two integrals in the above expression, we have
\[ \frac 1{\mathcal{L}} \int\limits_{0}^{\infty} \widehat \phi\left( 1+\tau/\mathcal{L} \right)  \bigg(  e^{\tau/2}N(l)\sum_{\substack{ j \in
   \mz[i] \\ j \neq 0}}\left ( \widetilde{g}\left(2\sqrt {N(j) e^{\tau/2}N(l)}\right )-\widetilde{g}\left(\sqrt {N(j) e^{\tau/2}N(l)}\right)\right )\, \dif \tau  \ll \frac {1}{\mathcal{L}} \int\limits_{0}^{\infty} \widehat \phi\left( 1+\tau/\mathcal{L} \right)  \, \dif \tau \ll 1, \]
   and
\[ \frac {1}{\mathcal{L}}   \int\limits_0^{\infty} \widehat \phi \left( 1-\tau/\mathcal{L} \right)    \sum_{\substack {k \in
   \mz[i] \\ k \neq 0}}\left ( \frac 12 g\left(\frac {N(k)}{2}\sqrt{\frac {e^{\tau}}{N(l^2)}}\right)-g\left(N(k)\sqrt{\frac {e^{\tau}}{N(l^2)}}\right)\right )  \dif \tau \ll \frac {1}{\mathcal{L}}   \int\limits_0^{\infty} \widehat \phi \left( 1-\tau/\mathcal{L} \right)   N(l) e^{-\tau/2}  \dif \tau \ll N(l). \]

    We can therefore use $\Phi(X)$ defined in \eqref{Phi} and partial summation to extend the sum over $l$ to all odd elements in $\mathcal{O}_K$ by introducing an extra error term of size $O\big(X^{\varepsilon}Z^{-3/2}\big)$. We now set $Z=X^{2/3 - \sigma/3}$,  change the order of summation in \eqref{Soddestim} and apply \eqref{W} to derive the desired result.
\end{proof}

\section{Proof of Theorem \ref{quadraticmainthm}}

   We combine Lemma \ref{lemma logd}, Lemma \ref{lemma explicit formula all d}, Lemma \ref{lemma Seven} and Lemma \ref{LemSodd} to arrive at the following
\begin{lemma}\label{lemma precise}
 Suppose that GRH is true. Let $\phi(x)$ be an even Schwartz function whose
Fourier transform $\hat{\phi}(u)$ has compact support in $(-2,2)$ and let $w$ be an even non-zero and non-negative Schwartz function.  For any integer $M \geq 1$, the $1$-level density of low-lying zeros in the family $\mathcal F$ of quadratic Hecke $L$-functions is given by
\begin{align}
\label{Theorem 3.5}
\begin{split}
\mathcal D(\phi;w, X) = \widehat \phi(0)-&\frac{1}2\int\limits_{-1}^{1} \widehat \phi(u)\,du +\frac{\widehat \phi(0)}{\mathcal{L}}  \bigg( \log\frac {32}{\pi^2}+ 2\frac
    {\Gamma'}{\Gamma}\left( \frac 12 \right) + \frac 2{\widehat w(0)}\int\limits_0^{\infty} w(x) \log x\, \dif x \bigg) +J(X) \\
    &+\frac{2}{\mathcal{L} }\int\limits_0^\infty\frac{e^{-x/2}}{1-e^{-x}}\left(\widehat{\phi}(0)-\widehat{\phi}\left(\frac{x}{\mathcal{L} }\right)\right) \dif x + \sum_{m=1}^M \frac{d_m \widehat \phi^{(m-1)}(0)}{\mathcal{L} ^m} +O\left( \frac 1{\mathcal{L} ^{M+1}} \right),
\end{split}
\end{align}
   where $J(X)$ is given as in Lemma \ref{LemSodd} and the coefficients $d_k$ are explicitly computable numbers given in Lemma \ref{lemma Seven}.
\end{lemma}

   The next lemma allows us to expand $J(X)$ in descending powers of $\mathcal{L} = \log X$.  This is a generalization of \cite[Lemma 3.6]{FPS}.
\begin{lemma}\label{lemma expansion}
Suppose that GRH is true and suppose that $\sigma=\sup(\text{supp } \widehat \phi)< 2$. Then for any integer $M\geq 1$, we have the expansion
\begin{align}
\label{Jexpansion}
 J(X)=\sum_{m=1}^M \frac{c_{w,m} \widehat \phi^{(m-1)}(1)}{\mathcal{L}^m}+ O\left( \frac 1{\mathcal{L} ^{M+1}} \right),
\end{align}
 where the constants $c_{w,m}$ can be given explicitly.
\end{lemma}

\begin{proof}
   Note that as  $\sigma=\sup(\text{supp } \widehat \phi)< 2$, we have
\begin{align*}
   J(X)= & \frac 1{\mathcal{L}} \int\limits_{0}^{\mathcal{L} } \bigg( \widehat \phi \left( 1+ \tau/\mathcal{L} \right)   \sqrt{2}e^{\tau/2} \sum_{\substack{ j \in
   \mz[i] \\ j \neq 0}} h_1\big(N(j)e^{\tau/2}\big)  +\widehat \phi \left( 1-\tau/\mathcal{L}  \right)   \sum_{\substack {k \in
   \mz[i] \\ k \neq 0}} h_2\big(N(k)e^{\tau/2}\big)\bigg)\, \dif \tau.
\end{align*}

   Recall that we have the bounds $h_1(x)\ll x^{-N}$ for any $N\geq 1$ and $h_2(x)\ll x^{-3/2+\varepsilon}$ for any $\varepsilon >0$ under GRH.
It follows from this that we can expand $\widehat \phi$ in Taylor series to obtain that
\begin{align}
\label{Jexpand}
\begin{split}
J(X)=& \sum_{m=1}^M \frac {\widehat \phi^{(m-1)}(1)}{(m-1)!\mathcal{L} ^m} \int\limits_{0}^{\mathcal{L} } \bigg(  \tau^{m-1} e^{\tau/2} \sum_{\substack{ j \in
   \mz[i] \\ j \neq 0}} h_1\big(N(j)e^{\tau/2}\big)   + (-\tau)^{m-1} \sum_{\substack {k \in
   \mz[i] \\ k \neq 0}} h_2\big(N(k)e^{\tau/2}\big)\bigg)\, \dif \tau \\
   &\hspace*{3cm} +O\big(\mathcal{L} ^{-M-1}\big),
\end{split}
\end{align}
   since here the error term introduced can be estimated as
\begin{align*}
 \ll \mathcal{L} ^{-M-1} \int\limits_{0}^{\mathcal{L} } \bigg(  \tau^{M}e^{-\tau/2}  + (-\tau)^{M}e^{-(3/4+\varepsilon)\tau} \bigg)\, \dif \tau
 \ll \mathcal{L} ^{-M-1} \int\limits_{0}^{\infty} \bigg(  \tau^{M}e^{-\tau/2}  + (-\tau)^{M}e^{-(3/4+\varepsilon)\tau} \bigg)\, \dif \tau \ll \mathcal{L} ^{-M-1}.
\end{align*}

   We now extend the integral in \eqref{Jexpand} to infinity and note that the error introduced by this extension can be easily shown to be
\[   \ll \sum_{m=1}^M \frac {\widehat \phi^{(m-1)}(1)}{(m-1)!\mathcal{L} ^m} \int\limits^{\infty}_{\mathcal{L}} \bigg(  \tau^{m-1}e^{-\tau}+ (-\tau)^{m-1} e^{-(3/4+\epsilon)\tau} \bigg)\, \dif \tau  \ll \mathcal{L} ^{-M-1}. \]

     We then deduce from this and \eqref{Jexpand} that
\begin{align*}
J(X)=&\sum_{m=1}^M \frac {\widehat \phi^{(m-1)}(1)}{(m-1)!\mathcal{L} ^m} \int\limits_{0}^{\infty} \bigg(  \tau^{m-1}\sqrt{2}e^{\tau/2} \sum_{\substack{ j \in
   \mz[i] \\ (j, 1+i)=1}} h_1\big(N(j)e^{\tau/2}\big)   + (-\tau)^{m-1} \sum_{\substack {k \in
   \mz[i] , k \neq 0}}(-1)^{N(k)} h_2\big(N(k)e^{\tau/2}\big)\bigg)\, \dif \tau \\
   &\hspace*{1cm} +O\big(\mathcal{L} ^{-M-1}\big),
\end{align*}

   As the integral in the above expression converges, the assertion of the lemma now follows from this.
\end{proof}

  We now substitute \eqref{Jexpansion} into \eqref{Theorem 3.5} and expand $\widehat{\phi}\left(x/\mathcal{L} \right)$ into its Taylor series around $0$ so that
\begin{align*}
\widehat{\phi}\left(x/\mathcal{L} \right)=&\sum_{m=0}^M \frac {\widehat \phi^{(m)}(0)}{(m)!\mathcal{L} ^m} \int\limits_{0}^{\infty} \bigg(  \tau^{m-1}\sqrt{2}e^{\tau/2} \sum_{\substack{ j \in
   \mz[i] \\ (j, 1+i)=1}} h_1\big(N(j)e^{\tau/2}\big)   + (-\tau)^{m-1} \sum_{\substack {k \in
   \mz[i] , k \neq 0}}(-1)^{N(k)} h_2\big(N(k)e^{\tau/2}\big)\bigg)\, \dif \tau \\
   &\hspace*{1cm} +O\big(\mathcal{L} ^{-M-1}\big),
\end{align*}
  and an interchange of the series and the integral allow us to deduce \eqref{Dexpansion}. In particular, we see that we have for $m \geq 2$,
\begin{align}
\label{Rwm}
R_{w,m}(\phi)=c_{w,m} \widehat \phi^{(m-1)}(1)+d_m\widehat \phi^{(m-1)}(0)-\frac {2\widehat \phi^{(m-1)}(0)}{(m-1)!}\int\limits_0^\infty\frac{e^{-x/2}x^{m-1}}{1-e^{-x}}\dif x.
\end{align}
   This completes the proof of Theorem \ref{quadraticmainthm}.

\section{Proof of Theorem \ref{oneleveldensityresult}}
\label{The ratios conjecture's prediction}

   We first follow the recipe given in \cite{CFZ} to derive
a suitable version of the ratios conjecture for the family $\mathcal F$.  We start by considering the expression
\begin{equation}
\label{ralphgam}
R(\alpha,\beta)=\frac1{W(X)}\sumn \frac{L\left(1/2+\alpha,\chi_{i(1+i)^5c}\right)}{L\left(1/2+\beta,\chi_{i(1+i)^5c}\right)}.
\end{equation}

 Similar to the treatment in \cite[Section 4.1]{G&Zhao6}, we may approximate $L(s, \chi_{i(1+i)^5c})$ by
\begin{equation}
\label{approxeq}
L\left(s, \chi_{i(1+i)^5c}\right)\approx \sum_{\mathfrak{n} \neq 0} \frac{\chi_{i(1+i)^5c}(\mathfrak{n})}{N(\mathfrak{n})^s}+X_{c}(s)\sum_{\mathfrak{n} \neq 0}
\frac{\chi_{i(1+i)^5c}(\mathfrak{n} )}{N(\mathfrak{n})^{1-s}},
\end{equation}
where $\sum_{\mathfrak{n} \neq 0}$ denotes a sum over non-zero integral ideals in $\mathcal{O}_K$ and
\begin{equation}
\label{xe}
X_{c}(s)=\frac{\Gamma\left(1-s\right)}{\Gamma\left(s \right)}\left(\frac{\pi^2}{32N(c)}\right)^{s-1/2}.
\end{equation}

Writing $\mu_{[i]}$ for the M\"obius function on $K$, we obtain that for $\Re(s)>1$,
\begin{equation}
\label{lemobius}
\frac{1}{L(s,\chi_{i(1+i)^5c})} = \sum_{\mathfrak{m} \neq 0}
\frac{\mu_{[i]}(\mathfrak{m})\chi_{i(1+i)^5c}(\mathfrak{m})}{N(\mathfrak{m})^s}.
\end{equation}

  Applying \eqref{approxeq} and \eqref{lemobius} to $\eqref{ralphgam}$, we see that
\begin{equation}
\label{ragone1}
 R(\alpha, \beta) \approx  R_1(\alpha, \beta) +R_2(\alpha, \beta),
\end{equation}
   where
\[ R_1(\alpha, \beta) = \frac1{W(X)}\sumn \sum_{\mathfrak{m},\mathfrak{n}\neq 0} \frac{\mu_{[i]}(\mathfrak{m})\chi_{i(1+i)^5c}(\mathfrak{nm})}{N(\mathfrak{m})^{1/2+\beta}N(\mathfrak{n})^{1/2 +\alpha}},  \]
and
\[ R_2(\alpha, \beta)= \frac1{W(X)}\sumn X_{c}\left(\frac{1}{2} + \alpha\right)\sum_{\mathfrak{m},\mathfrak{n}\neq 0} \frac{\mu_{[i]}(\mathfrak{m})\chi_{i(1+i)^5c}(\mathfrak{nm})}{N(\mathfrak{m})^{1/2+\beta}N(\mathfrak{n})^{1/2 -\alpha}}. \]

 When $\mathfrak{nm}$ is an odd square, we expect to gain a main contribution to both $R_1$ and $R_2$.  Applying Lemma \ref{lemma count of squarefree}, we have in this case
\begin{align*}
\frac{1}{W(X)}\sumn\chi_{i(1+i)^5c}(\mathfrak{nm})\approx
\prod_{\substack{ \varpi \equiv 1 \bmod {(1+i)^3} \\ \varpi\mid \mathfrak{nm}}}  \left(1+\frac 1{N(\varpi)} \right)^{-1}.
\end{align*}
  We then deduce that, upon writing $\square$ for a perfect square,
\begin{equation*}
R_1(\alpha, \beta) \sim \widetilde R_1(\alpha, \beta)=
\sum_{\mathfrak{nm} = \text{odd } \square} \frac{\mu_{[i]}(\mathfrak{m})}{N(\mathfrak{m})^{1/2+\beta}N(\mathfrak{n})^{1/2 + \alpha}}\prod_{\substack{ \varpi \equiv 1 \bmod {(1+i)^3} \\ \varpi\mid \mathfrak{nm}}}  \left(1+\frac 1{N(\varpi)} \right)^{-1}.
\end{equation*}
 A computation on the Euler product shows that
\begin{align}
\label{ragtilde1}
\widetilde R_1(\alpha,\beta)= \frac{\zeta_K(1+2\alpha)}{\zeta_K(1+\alpha+\beta)}A(\alpha,\beta),
\end{align}
where
\begin{equation} \label{defnofae}
\begin{split}
A(\alpha,\beta) =\left(\frac{2^{1+\alpha+\beta}-2^{\beta-\alpha}}{2^{1+\alpha+\beta}-1}\right)\prod_{\varpi \equiv 1 \bmod {(1+i)^3}}
 & \left(1 -\frac{1}{N(\varpi)^{1+\alpha+\beta}}\right)^{-1} \\
 & \times \left(1-\frac{1}{(N(\varpi)+1)N(\varpi)^{1+2\alpha}} -\frac{1}{(N(\varpi)+1)N(\varpi)^{\alpha+\beta}}\right).
\end{split}
\end{equation}
Note that the product $A(\alpha,\beta)$ is absolutely convergent for $\Re(\alpha)$, $\Re(\beta)> -1/4$. \newline
	
 Similarly, we obtain
\begin{align}
\label{R2}
  R_2(\alpha, \beta) \approx  \widetilde R_2(\alpha, \beta)=\frac{1}{W(X)}\sumn X_{c}\left(\frac{1}{2} + \alpha\right)\widetilde R_1(-\alpha,\beta).
\end{align}

  Combining \eqref{ragone1} with \eqref{ragtilde1} and \eqref{R2}, we deduce the following appropriate version of the ratios conjecture for our family $\mathcal F$ .	
\begin{conjecture}
\label{ratiosconjecture}
 Let $\varepsilon>0$ and let $w$ be an even and nonnegative Schwartz test function on $\rear$ which is not identically zero. For complex numbers $\alpha$ and $\beta$ satisfying $|\Re(\alpha)|< 1/4$, $(\log X)^{-1} \ll \Re(\beta) < 1/4$ and $\Im(\alpha), \Im(\beta) \ll X^{1-\varepsilon}$, we have that
\begin{align*}
 \frac{1}{W(X)} & \sumn \frac{L\left( 1/2 +\alpha,\chi_{i(1+i)^5c} \right)}{L\left( 1/2 + \beta,\chi_{i(1+i)^5c} \right)} \\
&= \frac{\zeta_K(1+2\alpha)}{\zeta_K(1+\alpha+\beta)}A(\alpha,\beta)
+\frac{1}{W(X)}\sumn X_{c}\left(\frac{1}{2}+\alpha\right) \frac{\zeta_K(1-2\alpha)}{\zeta_K(1-\alpha+\beta)}A(-\alpha,\beta)+ O_\varepsilon\big(X^{-1/2+\varepsilon}\big),
\end{align*}
where $A(\alpha,\beta)$ is defined in $\eqref{defnofae}$ and $X_c(s)$ is defined in $\eqref{xe}$.
\end{conjecture}
	
 Similar to the derivation of \cite[Lemma 4.3]{G&Zhao6}, we deduce from Conjecture \ref{ratiosconjecture} the following result needed in the calculation of the $1$-level density.
\begin{lemma}
\label{ratiostheorem}
Assuming the truth of GRH and Conjecture \ref{ratiosconjecture}, we have for any $\varepsilon>0$, $(\log X)^{-1} \ll \Re(r)<1/4$ and $\Im(r)\ll X^{1-\varepsilon}$,
\begin{align*}
\begin{split}
&\frac{1}{W(X)}\sumn \frac{L' \left( 1/2 + r,\chi_{i(1+i)^5c} \right)}{L \left( 1/2 + r,\chi_{i(1+i)^5c} \right)} \\
&= \frac{\zeta'_K(1+2r)}{\zeta_K(1+2r)}+A_{\alpha}(r,r)-\frac 4{\pi}\frac{1}{W(X)}\sumn_{(c,1+i)=1}  X_{c}\left(\frac{1}{2}+r\right) \zeta_K(1-2r)A(-r,r)+ O_{\varepsilon}\big(X^{-1/2+\varepsilon}\big),
\end{split}
\end{align*}
  where
\begin{equation*}
A_{\alpha}(r,r) = \frac{\partial}{\partial \alpha}A(\alpha,\beta)\bigg|_{\alpha=\beta=r}.
\end{equation*}
\end{lemma}
	
   We now proceed as in \cite[Section 4.4]{G&Zhao6}.  Assuming the truth of GRH, it follows from Lemma \ref{ratiostheorem} that
\begin{align}
\label{RATIOSONELEV}
\begin{split}
D(\phi;w, X)= \frac{1}{W(X)}\sumn & \frac{1}{2\pi i}\int\limits_{(a-1/2)}  \bigg(2 \frac{\zeta'_K(1+2r)}{\zeta_K(1+2r)}+2A_{\alpha}(r,r)-\frac{X_{c}'\left(1/2+r\right)}{X_{c}\left(1/2+r\right)} \\
& -\frac {8}{\pi} X_{c}\left(\frac{1}{2}+r\right)\zeta_K(1-2r)A(-r,r)\bigg) \phi\left(\frac{i \LL r}{2\pi}\right) \dif r + O_{\varepsilon} \left( X^{-1/2+\varepsilon} \right),
\end{split}
\end{align}
  where $D(\phi;w, X)$ is defined in \eqref{D}. Note that the integrand in \eqref{RATIOSONELEV} is analytic in the region $\Re(r) \geq  0$ (in particular it is analytical at $r = 0$).  The assertion of Theorem \ref{oneleveldensityresult} now follows by moving the contour of integration from $\Re(r)=a-1/2$ to
$\Re(r)=0$.

\section{Proof of Theorem \ref{quadraticmainthmratio}} \label{Compartionofresults}

\subsection{Initial treatment}
\label{section compare results}

  In this section, we consider the expansions of the $D(\phi;w, X)$ given in \eqref{Drationconj} as powers of $1/\LL$ with $\LL=\log X$. Recall from \eqref{RATIOSONELEV} that, up to an error term of size $O(\LL^{-2})$, we have
\begin{align}
\label{precomp}
\begin{split}
  D(\phi;w, X)
=\frac{1}{W(X)}\sumn & \frac{1}{2\pi i}\int\limits_{(a')}  \bigg(2 \frac{\zeta'_K(1+2r)}{\zeta_K(1+2r)}+2A_{\alpha}(r,r)-\frac{X_{c}'\left(\frac12+r\right)}{X_{c}\left(\frac12+r\right)}\\
& -\frac {8}{\pi} X_{c}\left(\frac{1}{2}+r\right)\zeta_K(1-2r)A(-r,r)\bigg) \phi\left(\frac{i \LL r}{2\pi}\right) \dif r ,
\end{split}
\end{align}
 where $\LL^{-1}<a'< 1/4$. \newline

  We set
\begin{align} \label{I}
 I=-\frac {8}{\pi} \frac{1}{W(X)}\sumn \frac{1}{2\pi i}\int\limits_{(a')}  \bigg(
 X_{c}\left(\frac{1}{2}+r\right)\zeta_K(1-2r)A(-r,r)\bigg) \phi\left(\frac{i \LL r}{2\pi}\right) \dif r.
\end{align}
  We shall postpone the evaluation of $I$ in the next section and proceed here the treatment on the other terms on the right-hand side of \eqref{precomp}. \newline

  We deduce first from by Lemma \ref{lemma count of squarefree} and \eqref{W}, after partial summation, that
\begin{align}
\label{thirdterm}
\begin{split}
\frac{1}{W(X)} \sumn & \frac{1}{2\pi}
\int\limits_{\mathbb R} \log\left(\frac{32N(c)}{\pi^2}\right) \phi\left(\frac{t \LL}{2\pi}\right) \, \dif t \\
=& \frac{\widehat \phi(0)}{\LL}\frac{1}{W(X)}\sumn \log\left(\frac{32N(c)}{\pi^2}\right) \\
=& \frac{\widehat \phi(0)}{\LL}\Big ( \log \frac{32}{\pi^2}+\LL+ \frac 2{\widehat w(0)}\int\limits_0^{\infty} w(x) \log x \ \dif x  \Big )+O \left( \LL^{-2} \right).
\end{split}
\end{align}
  Next note that, similar to \cite[(4.33)]{G&Zhao6}, we have
\begin{align}
\label{fifthterm}
\begin{split}
 \frac{1}{{W}(X)}  \sumn \frac{1}{2\pi} &
\int\limits_{\mathbb R}\left( \frac{\Gamma'}{\Gamma}\left(\frac 12-it\right)
+\frac{\Gamma'}{\Gamma}\left(\frac 12+it  \right)\right ) \phi\left(\frac{t\LL }{2\pi}\right) \, \dif t\\
&= 2\frac{\Gamma'}{\Gamma}\left(\frac 12\right)\frac{\widehat{\phi}(0)}{\LL} +\frac{2}{\LL} \int\limits_0^\infty\frac{e^{-t/2}}{1-e^{-t}}\left(\widehat{\phi}(0)-\widehat{\phi}\left(\frac{t}{\LL}\right)\right) \dif t.
\end{split}
\end{align}

 Furthermore, we follow the treatment of \cite[Lemma 4.1]{FPS1} to obtain via a direct calculation (noting that $A(r,r)=1$) that
\begin{equation*}
A_\alpha(r,r)+\frac{\zeta'_K(1+2r)}{\zeta_K(1+2r)}=\frac{-1}{N(\varpi) +1} \sum_{\varpi \equiv 1 \bmod {(1+i)^3}} \frac{N(\varpi)\log N(\varpi)}{N(\varpi)^{1+2r}-1}.
\end{equation*}
  It follows from this, after the substitution $u=-i\LL r/(2\pi)$ and interchanging the summations and the integral, that
\begin{align} \label{ANOTHERONELEVELDENSITY}
\begin{split}
\frac{1}{2\pi i}\int\limits_{(a')}   & \bigg(2 \frac{\zeta'_K(1+2r)}{\zeta_K(1+2r)}+2A_{\alpha}(r,r) \bigg) \phi\left(\frac{i\LL r}{2\pi}\right) \dif r \\
=& -\frac{2}{\LL}\sum_{\varpi \equiv 1 \bmod {(1+i)^3}} \frac{N(\varpi)\log N(\varpi)}{N(\varpi)+1} \sum_{j=1}^{\infty} \frac{1}{N(\varpi)^{j}}\int\limits_{\mathcal C'}\phi\left(u\right)\exp \left( -2\pi i u\left(\frac{2j\log N(\varpi)}{\LL}\right) \right) \dif u.
\end{split}
\end{align}
 where $\mathcal C'$ denotes the horizontal line $\Im(u)=-\LL a'/(2\pi)$. \newline

 As $\widehat \phi$ is compactly supported and $\phi(z)=\int_{\rear}\widehat\phi(x)e^{2\pi ixz}\, \dif x$, it follows from integration by parts that uniformly for $-\LL c'/(2\pi) \leq t\leq 0$,
\begin{align*}
\left|\phi(T+it)\right| \ll \frac1{|T|+1}.
\end{align*}
   In view of this, we can shift the contour of the last integration in \eqref{ANOTHERONELEVELDENSITY} from $\mathcal C'$ to $\Im(u) =0$ to deduce that
\begin{align}
\label{equation lemma 4.1}
\begin{split}
\frac{1}{2\pi i}\int\limits_{(a')} & \bigg(2 \frac{\zeta'_K(1+2r)}{\zeta_K(1+2r)}+2A_{\alpha}(r,r)\bigg) \phi\left(\frac{i\LL r}{2\pi}\right) \dif r	\\
& =- \frac 2{\LL}\sum_{\substack{\varpi \equiv 1 \bmod {(1+i)^3} \\ j\geq 1}} \frac{\log N(\varpi)}{N(\varpi)^{j}} \left(  1+\frac 1{N(\varpi)} \right)^{-1} \widehat \phi\left( \frac{2j \log N(\varpi)}{\LL} \right).
\end{split}
\end{align}

\subsection{Evaluation of $I$}
  In this section, we evaluate $I$, defined in \eqref{I}. Our treatment here follows from the proof of \cite[Lemma 4.6]{FPS1}. We deduce from \eqref{defnofae} that
\begin{align*}
A(-\gamma,\gamma)=\frac {3(2-2^{2r})}{4-2^{2r}}\frac {\zeta_K(2)}{\zeta_K(2-2r)}.
\end{align*}
  Substituting the above in the right-hand side of \eqref{I}, we deduce from the definitions of $X_c$ given in \eqref{xe} and a change of variable $r=2\pi i \tau / \LL$ that
\begin{align} \label{Big I}
\begin{split}
I= -\frac {8}{\pi} \frac{\zeta_K(2)}{\LL}\int\limits_{\mathcal C'} &
\Bigg(\frac{\Gamma\left(1/2 - 2\pi i \tau/\LL\right)}{\Gamma\left(1/2 +2\pi i \tau/\LL \right)} \Bigg)\left(\frac{\pi^2}{32}\right)^{2 \pi i \tau/\LL}\left(1+\frac{2-2^{4\pi i \tau/\LL +1}}{4-2^{4\pi i \tau/\LL}}\right)\frac{\zeta_K\left(1-4\pi i \tau/\LL \right)}{\zeta_K\left(2- 4\pi i \tau/\LL\right)} \phi\left(\tau\right)\\
&\times\frac{1}{W(X)}\sumn N(c)^{- 2\pi i \tau/\LL} \dif \tau,
\end{split}
\end{align}
where we also denote $\mathcal C'$ for the horizontal line $\Im(\tau)=- \LL a'/(2\pi)$. \newline

We treat the last sum in \eqref{Big I} by applying Mellin inversion to see that for $0\leq \Re(r) \leq 1/2$,
$$  \sumn N(c)^{-r}= \frac 4{2\pi i} \int\limits_{(2)}\frac{2^{s+r}}{2^{s+r}+1} \frac{\zeta_K(s+r)}{\zeta_K(2(s+r))} X^s \mathcal M w(s) \ \dif s. $$
 We shift the contour of integration to the line $\Re(s)=1/2 -\Re(r)+\varepsilon$ to encounter a simple pole at $s=1-r$. On the new line of integration, the convexity bound (see \cite[Exercise 3, p. 100]{iwakow}), together with the rapid decay of $\mathcal M w$, gives
\begin{align*}
  \zeta_K(s) \ll (1+|s|^2)^{1/4+\varepsilon}.
\end{align*}
With this and recalling that the reside of $\zeta_K(s)$ at $s=1$ is $\pi/4$, we get
$$ \sumn N(c)^{-r}= \frac{2\pi}{3\zeta_K(2)}X^{1-r}\mathcal M w(1-r)  +O_{\varepsilon,w}\left(\big(|\Im(r)|+1\big)^{1/2+\varepsilon}X^{1/2-\Re(r)+\varepsilon}\right). $$
Combining the above with \eqref{W}, we deduce that for any $\varepsilon>0$ and $0\leq \Re(r) \leq 1/2$,
\begin{equation*}
\frac{1}{W(X)}\sumn N(c)^{-r} = \frac{2}{\widehat{w}(0)}X^{-r}\mathcal M w(1-r) +O_{\varepsilon,w}\left(\big(|\Im(r)|+1\big)^{1/2+\varepsilon}X^{-1/2-\Re (r)+\varepsilon}\right).
\end{equation*}
	
 For small $\varepsilon, \eta>0$, we change the contour $\mathcal C'$ in \eqref{Big I} to the path
$$ C=C_0\cup C_1\cup C_2, $$
where
$$C_0= \{ \Im(\tau)= 0 , |\Re(\tau)| \geq \LL^\varepsilon \},
\quad C_1= \{ \Im(\tau) = 0 ,  \eta \leq |\Re(\tau)| \leq \LL^\varepsilon \}, \quad C_2= \{ |\tau| = \eta, \Im(\tau) \leq 0 \}.  $$	
 As $\phi$ decays rapidly, the integration of $I$ over $C_0$ can be shown to be negligible.
We now apply the Taylor expansion to treat the integration of $I$ over $C_1\cup C_2$ by noting that
\begin{align*}
 \frac{\Gamma\left(1/2 - 2\pi i \tau/\LL \right)}{\Gamma\left(1/2 + 2\pi i \tau/\LL \right)} =1-2\frac {\Gamma'(1/2)}{\Gamma(1/2)}\frac{2\pi i \tau}\LL +O\bigg( \frac{|\tau|^2}{\LL^2}\bigg),
\end{align*}
  and that (see \cite[Formula 2, Section 8.366]{GR})
\begin{align*}
 -\frac {\Gamma'(1/2)}{\Gamma(1/2)}=2\log 2+\gamma.
\end{align*}

Using Taylor expansion and \eqref{zetaKexpan}, we get
\[   \left(1+\frac{2-2^{4\pi i \tau/\LL+1}}{4-2^{4\pi i \tau/\LL}}\right)\frac{1}{\zeta_K\left(2- 4\pi i \tau/\LL\right)}
  = \frac {1}{\zeta_K(2)}+\left ( \frac {1}{\zeta_K(2)} \left (-\frac {2\log 2}{3} \right) +\frac {\zeta'_K(2)}{\zeta_K^2(2)} \right )\frac{4\pi i \tau}\LL+O\bigg( \frac{|\tau|^2}{\LL^2}\bigg), \]
  and
  \[   \zeta_K\left(1-\frac{4\pi i \tau}\LL\right)= -\frac {\pi}{4}\cdot \frac \LL{4\pi i \tau}+\gamma_K  +O\bigg( \frac{|\tau|}{\LL}\bigg) . \]

Using the above formulas, we get, after a short computation,
\begin{align*}
&-\frac {8}{\pi} \frac{\zeta_K(2)}{\LL}
\frac{\Gamma\left( \frac{1}{2} -\frac{2\pi i \tau}{\LL} \right)}{\Gamma\left(\frac{1}{2} + \frac{2\pi i \tau}{\LL} \right)} \left(\frac{\pi^2}{32}\right)^{\frac{2 \pi i \tau}{\LL}}\left(1+\frac{2-2^{\frac{4\pi i \tau}\LL+1}}{4-2^{\frac{4\pi i \tau}\LL}}\right)\frac{\zeta_K\left(1-\tfrac{4\pi i \tau}\LL\right)}{\zeta_K\left(2-\tfrac{4\pi i \tau}\LL\right)} \phi\left(\tau\right)\frac{1}{W(X)}\sumn N(c)^{-\frac{2\pi i \tau}\LL} \\
&=    \frac 1{2\pi i\tau} \bigg( 1+ \frac{2\pi i \tau }\LL\bigg( 2\gamma +2 \log 4 + \log \left( \frac {\pi^2}{32} \right)+2\frac{\zeta'_K(2)}{\zeta_K(2)} -\frac{4}{3} \log 2 -\frac {8}{\pi} \gamma_K -\frac{\mathcal M w'(1)}{\mathcal Mw(1)} \bigg) +O\bigg( \frac{|\tau|^2}{\LL^2}\bigg)\bigg)\phi(\tau) e^{-2\pi i \tau} \\
& \hspace*{8cm} +O_{\varepsilon,w}\big(X^{-1/2+\varepsilon}\big).
\end{align*}

   We then deduce that
$$ I = \frac{1}{2\pi i} \int\limits_{C_1\cup C_2} \frac {\phi(\tau)}{\tau}  e^{-2\pi i \tau} \dif \tau+I'+ O_{w}(\LL^{-2}),$$
  where, combining the logarithm terms,
\begin{align*}
I' =&\frac 1{\LL} \bigg( 2\gamma+ \log \left( \frac {\pi^2}{2^{7/3}} \right)+2\frac{\zeta'_K(2)}{\zeta_K(2)}-\frac {8}{\pi} \gamma_K -\frac{\mathcal M w'(1)}{\mathcal Mw(1)}  \bigg) \int\limits_{C_1\cup C_2}\phi(\tau) e^{-2\pi i \tau} \dif \tau \\
 =&   \frac 1{\LL} \bigg( 2\gamma+\log \left( \frac {\pi^2}{2^{7/3}} \right)+2\frac{\zeta'_K(2)}{\zeta_K(2)} -\frac {8}{\pi} \gamma_K -\frac{\mathcal M w'(1)}{\mathcal Mw(1)} \bigg) \int\limits_{\mathbb R}\phi(\tau) e^{-2\pi i \tau} \dif \tau+O(\LL^{-2}) \\
=&  \frac {\widehat \phi(1)}{\LL} \bigg( 2\gamma+\log \left( \frac {\pi^2}{2^{7/3}} \right) +2\frac{\zeta'_K(2)}{\zeta_K(2)} -\frac {8}{\pi} \gamma_K -\frac{\mathcal M w'(1)}{\mathcal Mw(1)} \bigg) +O(\LL^{-2}).
\end{align*}	

  Similar to the treatment of $I_1$ in the proof of \cite[Lemma 4.6]{FPS1}, we have
\begin{align*}
 \int\limits_{C_1\cup C_2} \frac 1{2\pi i\tau} \phi(\tau) e^{-2\pi i \tau} \dif \tau & =\int\limits^{\infty}_1 \widehat \phi(\tau) \dif \tau +O(\LL^{-2}).
\end{align*}	

  Thus, we conclude that
\begin{align*}
I = \int\limits^{\infty}_1 \widehat \phi(\tau) \dif \tau+\frac {\widehat \phi(1)}{\LL} \bigg( 2\gamma+\log \left( \frac {\pi^2}{2^{7/3}} \right)+2\frac{\zeta'_K(2)}{\zeta_K(2)} -\frac {8}{\pi} \gamma_K -\frac{\mathcal M w'(1)}{\mathcal Mw(1)} \bigg)  +O(\LL^{-2}).
\end{align*}

   Combining the above expression for $I$ and \eqref{precomp}--\eqref{fifthterm}, \eqref{equation lemma 4.1} together, we deduce that the expression  \eqref{Dexpansionratio} is valid.

\subsection{Comparing terms}

   In this section we show that the expression given in \eqref{Dexpansionratio} is in agreement with that given in \eqref{Sstar} when $\sigma=\text{sup}(\text{supp } \widehat \phi)<2$. In fact,  applying \eqref{thirdterm} and \eqref{fifthterm} in \eqref{Drationconj} and comparing it with \eqref{Sstar}, we see that, with the help of Lemma \ref{lemma count of squarefree}, it suffices to show that
\begin{align}
\label{precomp1}
\begin{split}
&-\frac 2{\LL W(X)} \underset{(c, 1+i) =1 }{\sum \nolimits^{*}}  w\left( \frac {N(c)}X \right) \sum_{j \geq 1} S_j(\chi_{i(1+i)^5c},\LL;\hat{\phi}) \\
=& \frac{1}{W(X)}\sumn \frac{1}{2\pi}
\int\limits_{\mathbb R} \bigg(2 \frac{\zeta'_K(1+2it)}{\zeta_K(1+2it)} + 2A_{\alpha}(it,it)-\frac {8}{\pi}  X_{c}\left(\frac{1}{2}+it\right)\zeta_K(1-2it)A(-it,it) \bigg) \,
\phi\left(\frac{t\LL}{2\pi}\right) \, \dif t \\
=& \frac{1}{W(X)}\sumn \frac{1}{2\pi i}\int\limits_{(a')}  \bigg(2 \frac{\zeta'_K(1+2r)}{\zeta_K(1+2r)}+2A_{\alpha}(r,r)
-\frac {8}{\pi} X_{c}\left(\frac{1}{2}+r\right)\zeta_K(1-2r)A(-r,r)\bigg) \phi\left(\frac{i\LL r}{2\pi}\right) \dif r,
\end{split}
\end{align}
 where $(\log X)^{-1}<a'< 1/4$. \newline

  Now, similar to the treatment in Section~\ref{anaprime},  we write
\begin{align}
\label{Sdecomp}
-\frac 2{\LL W(X)} \underset{(c, 1+i) =1 }{\sum \nolimits^{*}}  w\left( \frac {N(c)}X \right) \sum_{j \geq 1}S_j(\chi_{i(1+i)^5c},\LL;\hat{\phi})  = S_{\text{odd}}+S_{\text{even}},
\end{align}
  where
\[ S_{\text{odd}}=-\frac 2{\LL W(X)} \underset{(c, 1+i) =1 }{\sum \nolimits^{*}} w\left( \frac {N(c)}X \right) \sum_{\substack{j \geq 1 \\ j \equiv 1 \pmod 2}} S_j(\chi_{i(1+i)^5c},\LL ;\hat{\phi}), \]
and
\[ S_{\text{even}}= -\frac 2{\LL W(X)} \underset{(c, 1+i) =1 }{\sum \nolimits^{*}} w\left( \frac {N(c)}X \right) \sum_{\substack{j \geq 1 \\ j \equiv 0 \pmod 2}} S_j(\chi_{i(1+i)^5c},\LL ;\hat{\phi}). \]

  We then deduce from \eqref{equation lemma 4.1}, \eqref{precomp1}, \eqref{Sdecomp} and \eqref{Seven} that it remains to show that
\begin{align}
\label{seccomp}
&S_{\text{odd}}
=I +O(\LL^{-2}),
\end{align}
  where $I$ is defined in \eqref{I}.

\subsection{Evaluation of $S_{\text{odd}}$ }

  In this section, we evaluation $S_{\text{odd}}$ to the first lower order term. Our treatment here follows largely the approach in the proof of \cite[Theorem 1.1]{FPS1}. We recall from \eqref{LemSodd} that
\begin{align}
\label{Sodddecomp}
 S_{\text{odd}}  =   \int\limits_{1}^{\infty} \widehat \phi(u)\, \dif u+ J(X)+O\big(\LL^{-2} \big),
\end{align}
   where
\begin{align*}
 J(X)  =    \frac 1{\LL} \int\limits_{0}^{\infty} \bigg( \widehat \phi( 1+ \tau/\LL  )   e^{\tau/2} \sum_{\substack{ j \in
   \mz[i] \\  j \neq 0}} h_1\big(N(j)e^{\tau/2}\big)  +\widehat \phi \left( 1-\tau/\LL  \right)   \sum_{\substack {k \in
   \mz[i] \\ k \neq 0}} h_2\big(N(k)e^{\tau/2}\big)\bigg)\, \dif \tau,
\end{align*}

  We now evaluate $h_1(x)$ by applying the Mellin inversion to recast it as
\begin{align*}
h_1(x) &= \frac{3\zeta_K(2)}{\pi \widehat w(0)}\frac 1{2 \pi i} \int\limits_{(5/2)}\sum_{\substack{ l \equiv 1 \bmod {(1+i)^3}}}(2^{-z}-1)\frac {\mu_{[i]}(l)}{N(l)^{1+z}} \mathcal M g_1(z)\frac{\dif z}{x^{z}}  \\
&= \frac{3\zeta_K(2)}{\pi \widehat w(0)}\frac 1{2 \pi i} \int\limits_{(5/2)} \frac{2^{-z}-1}{( 1- 2^{-1-z})\zeta_K(1+z)}
\mathcal M g_1(z)\frac{\dif z}{x^{z}}.
\end{align*}
Similarly,	we have, with a change of variables $z \to -z$,
\begin{align*}
h_2 (x) &= \frac{3\zeta_K(2)}{\pi \widehat w(0)}\frac 1{2 \pi i} \int\limits_{(1/2)}\sum_{\substack{l \equiv 1 \bmod {(1+i)^3}}}(2^{z-1}-1)\frac {\mu_{[i]}(l)}{N(l)^{2-z}} \mathcal M g(z)\frac{\dif z}{x^{z}}  \\
&= \frac{3\zeta_K(2)}{\pi \widehat w(0)}\frac 1{2 \pi i} \int\limits_{(-1/2)} \frac{2^{-z-1}-1}{( 1- 2^{-2-z})\zeta_K(2+z)}
\mathcal M g(-z)x^{z} \dif z \\
&= \frac{3\zeta_K(2)}{\pi \widehat w(0)}\frac 1{2 \pi i} \int\limits_{(-5/4)} \frac{2^{-z-1}-1}{( 1- 2^{-2-z})\zeta_K(2+z)}
\mathcal M g(-z)x^{z} \dif z .
\end{align*}

 With the above expressions for $h_1(x)$ and $h_2(x)$, we can write $J(x)$ given in \eqref{corollary Sodd in terms of J} as
\begin{align}
\label{equation preTaylor}
\begin{split}
 J(X)=& \frac{3\zeta_K(2)}{\LL \pi \widehat w(0)} \frac 1{2 \pi i}\int\limits_{0}^{\infty} \bigg( \widehat \phi \left( 1+\tau/\LL  \right) \int\limits_{(5/2)} \frac{4(2^{-z}-1)\zeta_K(z)}{( 1- 2^{-1-z})\zeta_K(1+z)} \mathcal Mg_1(z)\frac{\dif z}{e^{(z-1)\tau/2}} \\
 &\hspace*{3cm} +\widehat \phi( 1- \tau/\LL  )   \int\limits_{(-5/4)} \frac{ 4(2^{-z-1}-1)\zeta_K(-z)}{( 1- 2^{-2-z})\zeta_K(2+z)} \mathcal M g(-z)e^{z\tau/2} \dif z \bigg) \, \dif \tau \\
 =& \frac{3\zeta_K(2)}{\LL \pi \widehat w(0)}\frac 1{2 \pi i}\bigg( \int\limits_{(3/2)} \frac{4(2^{-z}-1)\zeta_K(z)}{( 1- 2^{-1-z})\zeta_K(1+z)} \mathcal M g_1 (z) \int\limits_{0}^{\infty}  \widehat \phi \left( 1+\tau/\LL  \right)e^{-(z-1)\tau/2} \dif \tau \dif z \\
 &\hspace*{3cm} +   \int\limits_{(-5/4)} \frac{4(2^{-z-1}-1)\zeta_K(-z)}{( 1- 2^{-2-z})\zeta_K(2+z)} \mathcal M g(-z) \int\limits_{0}^{\infty}\widehat\phi( 1- \tau/\LL  ) e^{z\tau/2} \dif \tau \dif z \bigg)\,.
\end{split}
\end{align}	

Now we consider the Taylor expansions, centered at $1$, of $\widehat\phi \left( 1+ \tau/\LL  \right)$ and $\widehat\phi \left( 1- \tau/\LL  \right)$ in \eqref{equation preTaylor}. By keeping only the constant terms, we see that their contribution to $J(X)$ equals, with another change of variables $z \to z+1$ in the first integral,
\begin{align}
\label{Jconst}
 \frac{3\zeta_K(2)}{\LL \pi \widehat w(0)}\frac{ 8\widehat \phi( 1)}{2 \pi i}\bigg( \int\limits_{(1/2)} \frac{(2^{-z-1}-1)\zeta_K(z+1)}{( 1- 2^{-2-z})\zeta_K(2+z)} \mathcal M g_1(z+1) \frac{ \dif z}{z}-  \int\limits_{(-5/4)} \frac{(2^{-z-1}-1)\zeta_K(-z)}{( 1- 2^{-2-z})\zeta_K(2+z)} \mathcal M g(-z) \frac{\dif z}z \bigg).
\end{align}
  We now shift the contour of the last integration to the line $\Re(z) = 1/2$. We apply Lemma \ref{Melliniden} to see that the quantity in \eqref{Jconst} equals
\begin{align*}
  \frac{3\zeta_K(2)}{\LL \pi \widehat w(0)}\frac{ 8\widehat \phi( 1)}{2 \pi i}R,
\end{align*}
  where $R$ is the residue of the function
\begin{align*}
\frac{(2^{-z-1}-1)\zeta_K(-z)}{( 1- 2^{-2-z})\zeta_K(2+z)}  \frac{\mathcal M g(-z)}z
\end{align*}
   at $z=0$.  We observe via integration by parts that
\begin{align*}
 \mathcal M g(-z) =\int\limits^{\infty}_0 g(t)t^{-z}\frac {\dif t}{t}=\int\limits^{\infty}_0 g(t) \dif \Big ( \frac {t^{-z}}{-z} \Big )=\frac 1{z}\int\limits^{\infty}_0t^{-z} g'(t) \dif t.
\end{align*}
  As
  \[ \lim_{z \rightarrow 0}\int\limits^{\infty}_0t^{-z} g'(t) \dif t=g(0)=\widehat{w}(0)>0, \]
it follows that $\mathcal M g(-z)$ has a pole at $z=0$.  We apply \eqref{w0} and get that around $z=0$,
\begin{align}
\label{Mgz0}
\begin{split}
 \mathcal M g(-z) =& \frac {-g(0)}{z}-\int\limits^{\infty}_0(\log t) g'(t) \dif t+O(z^2)=\frac {-\widetilde{w}(0)}{z}-\int\limits^{\infty}_0(\log t) g'(t) \dif t+O(z^2) \\
=& -\frac {\pi}{2}\frac {\widehat{w}(0)}{z}-\int\limits^{\infty}_0(\log t) g'(t) \dif t+O(z^2).
\end{split}
\end{align}

   Moreover, we have around $z=0$,
\begin{equation} \label{otherres1}
 \frac{2^{-z-1}-1}{1- 2^{-2-z}}= -\frac 23+\frac{-\log 2 (\frac 38)+\frac 12 (\log 2) \frac 14}{(1- 2^{-2})^2}z=-\frac 23-\frac 49 (\log 2)z
 \end{equation}
 and
 \begin{equation} \label{otherres2}
\frac{\zeta_K(-z)}{\zeta_K(2+z)}= \frac{\zeta_K(0)}{\zeta_K(2)}+\frac{-\zeta'_K(0)\zeta_K(2)-\zeta_K(0)\zeta'_K(2)}{\zeta^2_K(2)}z+O(z^2).
\end{equation}
Using \eqref{Mgz0}, \eqref{otherres1} and \eqref{otherres2}, we get
\begin{equation} \label{R}
\begin{split}
  R=-\frac {\pi}{2}\widehat{w}(0) \left( -\frac 23 \right) & \frac{-\zeta'_K(0)\zeta_K(2)-\zeta_K(0)\zeta'_K(2)}{\zeta^2_K(2)} \\
  & - \frac {\pi}{2}\widehat{w}(0) \left( -\frac 49 \log 2 \right) \frac{\zeta_K(0)}{\zeta_K(2)} -\frac 23\frac{\zeta_K(0)}{\zeta_K(2)} \left( -\int\limits^{\infty}_0(\log t) g'(t) \dif t \right).
  \end{split}
\end{equation}

  To further simplify $R$, we use the fact that $s\Gamma(s) = 1$ (see \cite[\S 10]{Da}) when $s = 0$  and the functional equation for $\zeta_K(s)$ (see \cite[Theorem 3.8]{iwakow}):
\begin{align*}
 \pi^{-s}\Gamma(s)\zeta_K(s) = \pi^{-(1-s)}\Gamma(1 -s)\zeta_K(1 -s)
\end{align*}
  to obtain that $\zeta_K(0) = -1/4$. \newline

  We further use the relation (see \cite[\S 10]{Da})
\begin{align*}
  \Gamma(s)\Gamma(1-s)=\frac {\pi}{\sin (\pi s)}
\end{align*}
   to derive that
\begin{align}
\label{Zetaexp}
  \zeta_K(1 -s)=\pi^{-2s}\Gamma(s)^2\sin(\pi s)\zeta_K(s)
\end{align}
   Applying \eqref{zetaKexpan}, we see that around $s=1$, we have
\begin{align*}
  \zeta_K(s)\sin (\pi s)=-\frac {\pi^2}{4}-\pi \gamma_K(s-1)+O((s-1)^2).
\end{align*}

  Using the above expansion and the fact that $\Gamma'(1)=-\gamma$ (see \cite[Formula 1, Section 8.366]{GR} by noting also that $\Gamma(1)=1$), we take the derivative on both sides of \eqref{Zetaexp} to see that
\begin{align*}
  -\zeta'_K(0)=-\frac {1}{\pi}\gamma_K+\frac {\gamma}{2}+\frac {\log \pi}{2}.
\end{align*}

    Now, inserting the values of $\zeta_K(0), -\zeta'_K(0)$ into \eqref{R}, together with a short calculation, we obtain that
\begin{align}
\label{JXexp}
\begin{split}
  J(X)=& \frac{3\zeta_K(2)}{\LL \pi \widehat w(0)}\frac{ 8\widehat \phi( 1)}{2 \pi i}R+O(\LL^{-2}) \\
=& \frac{8\widehat \phi( 1)}{\LL}\left(-\frac {\gamma_K}{\pi}+\frac {\gamma}{2}+\frac {\log \pi}{2}+\frac 14\frac {\zeta'_K(2)}{\zeta_K(2)}+ \left( \frac 23 \log 2 \right)\zeta_K(0)-\frac {1}{2\pi \widehat{w}(0)} \left( \int\limits^{\infty}_0(\log t) g'(t) \dif t \right) \right)+O(\LL^{-2}).
\end{split}
\end{align}

  We evaluate the last integral above by noticing that for small $\eta>0$, we have
\begin{align*}
\int\limits_0^{\infty} (\log x) g'(x) \dif x &=\int\limits_0^{\infty} \sqrt{2} \log x \widetilde w'(\sqrt{2}x) \dif x = \int\limits_0^{\infty} \log  \bigg(\frac x{\sqrt{2}}\bigg) \widetilde w'(x) \dif x \\
&= \widetilde w(0) \log (\sqrt{2})  +\int\limits_{\eta}^{\infty} (\log x) \widetilde w'(x) \dif x +O\big(\eta \log (\eta^{-1})\big) \\
\end{align*}
Now the above expression becomes, after integration by parts,
\[ \widetilde w(0) \log (\sqrt{2}) -\int\limits_{\eta}^{\infty} \frac{\widetilde w(x)-\widetilde w(0)I_{[0,1]}(x)}x \dif x +O\big(\eta \log (\eta^{-1})\big), \]
where $I_{[0,1]}$ is the characteristic function of the interval $[0,1]$. \newline
	
  By evaluating $\widetilde{w}(x)$ in polar coordinates, we see that
\begin{align*}
     \widetilde{w}(x) =& 4\int\limits^{\pi/2}_0\int\limits^{\infty}_0\cos (2\pi r x \sin \theta)w(r^2) \ r \dif r \dif \theta   .
\end{align*}

  It follows from this and by letting $\eta \rightarrow 0^+$ and using \cite[Example (e) on page 132]{V}, we obtain that
\begin{equation} \label{logxg'}
\begin{split}
\int\limits_0^{\infty} (\log x )g'(x) \dif x =& \widetilde w(0) \log (\sqrt{2}) -\int\limits_{0}^{\infty} \frac{\widetilde w(x)-\widetilde w(0)I_{[0,1]}(x)}x \dif x  \\
=& \widetilde w(0) \log (\sqrt{2}) -4\int\limits^{\infty}_0w(r^2)r\int\limits^{\pi/2}_0 \left( \int\limits^1_0 \frac {\cos (2\pi r x \sin \theta )-1}{x}  \dif x + \int\limits^{\infty}_1 \frac {\cos (2\pi r x \sin \theta )}{x}  \dif x \right) \dif \theta \dif r
\end{split}
\end{equation}

Now the inner-most integrals over $x$ become
\[=  \int\limits^{2\pi r  \sin \theta }_0 \frac {\cos (u)-1}{u}  \dif u + \int\limits^{\infty}_{2\pi r \sin \theta} \frac {\cos (u)}{u}  \dif u  =  \gamma +\log \big( 2\pi r \sin \theta \big ) . \]
Hence the expression in \eqref{logxg'} is
\[ \widetilde w(0) \log (\sqrt{2}) +\frac {\pi \gamma \widehat{w}(0)}{2} +\frac {\pi \log \pi \widehat{w}(0)}{2}+\frac {\pi \log 2 \widehat{w}(0)}{2}+\frac {\pi}{2} \int\limits^{\infty}_0w(r)\log r \ \dif r+ 2\int\limits^{\infty}_0w(r) \dif r\int\limits^{\pi/2}_0  \log \big( \sin \theta \big ) \dif \theta. \]
  As we have (see \cite[Formula 3, Section 4.224]{GR})
\begin{align*}
 \int\limits^{\pi/2}_0  \log \big( \sin \theta \big ) \dif \theta=-\frac {\pi}{2}\log 2.
\end{align*}

  We thus conclude that
\begin{align*}
 \int\limits_0^{\infty} (\log x )g'(x) \dif x
&= \frac {\pi \widehat{w}(0)}{4}  \log 2 +\frac {\pi \gamma \widehat{w}(0)}{2} +\frac {\pi \log \pi \widehat{w}(0)}{2}+ \frac {\pi}{2} \mathcal M w'(1).
\end{align*}
  Applying this to \eqref{JXexp}, we see that
$$ J(X)=\frac {\widehat \phi(1)}{\LL} \bigg( 2\gamma+2\log 4+\log \left( \frac {\pi^2}{32} \right)+2\frac{\zeta'_K(2)}{\zeta_K(2)}- \frac 43 \log 2 -\frac {8}{\pi} \gamma_K -\frac{\mathcal M w'(1)}{\mathcal Mw(1)} \bigg)  +O(\LL^{-2}). $$
  With the above expression for $J(X)$ and \eqref{Sodddecomp}, we conclude that the expression given in \eqref{seccomp} is valid and this completes the proof of Theorem \ref{quadraticmainthmratio}.

\vspace{0.1in}

\noindent{\bf Acknowledgments.} P. G. is supported in part by NSFC grant 11871082 and L. Z. by the FRG Grant PS43707 and the Goldstar Award PS53450 from the University of New South Wales (UNSW).  Parts of this work were done when P. G. visited UNSW in September 2019. He wishes to thank UNSW for the invitation, financial support and warm hospitality during his pleasant stay. Finally, the authors thank the anonymous referee for his/her very careful reading of this manuscript and many helpful comments and suggestions.

\bibliography{biblio}
\bibliographystyle{amsxport}

\vspace*{.5cm}

\noindent\begin{tabular}{p{8cm}p{8cm}}
School of Mathematical Sciences & School of Mathematics and Statistics \\
Beihang University & University of New South Wales \\
Beijing 100191 China & Sydney NSW 2052 Australia \\
Email: {\tt penggao@buaa.edu.cn} & Email: {\tt l.zhao@unsw.edu.au} \\
\end{tabular}

\end{document}